\documentclass[12pt,a4paper]{amsart}
\usepackage{amssymb,eucal}
\usepackage{tikz-cd}

\calclayout

\newcommand{\+}{\nobreakdash-}
\renewcommand{\:}{\colon}

\newcommand{\rarrow}{\longrightarrow}
\newcommand{\ot}{\otimes}

\newcommand{\bu}{{\text{\smaller\smaller$\scriptstyle\bullet$}}}
\newcommand{\lrarrow}{\mskip.5\thinmuskip\relbar\joinrel\relbar\joinrel
 \rightarrow\mskip.5\thinmuskip\relax}

\DeclareMathOperator{\Hom}{Hom}
\DeclareMathOperator{\Ext}{Ext}
\DeclareMathOperator{\Tor}{Tor}
\DeclareMathOperator{\Ker}{Ker}
\DeclareMathOperator{\coker}{coker}
\DeclareMathOperator{\pd}{pd}
\DeclareMathOperator{\fd}{fd}
\DeclareMathOperator{\Id}{Id}
\DeclareMathOperator{\id}{id}

\newcommand{\bb}{{\mathsf{b}}}

\newcommand{\modl}{{\operatorname{\mathsf{--mod}}}}
\newcommand{\modr}{{\operatorname{\mathsf{mod--}}}}
\newcommand{\bimod}{{\operatorname{\mathsf{--mod--}}}}

\newcommand{\co}{{\operatorname{\mathsf{-co}}}}
\newcommand{\ctra}{{\operatorname{\mathsf{-ctra}}}}
\newcommand{\adj}{{\operatorname{\mathsf{-adj}}}}

\newcommand{\R}{\mathfrak R}
\newcommand{\C}{\mathcal C}

\newcommand{\Ab}{\mathsf{Ab}}

\newcommand{\sA}{\mathsf A}
\newcommand{\sB}{\mathsf B}
\newcommand{\sC}{\mathsf C}
\newcommand{\sD}{\mathsf D}
\newcommand{\sG}{\mathsf G}
\newcommand{\sK}{\mathsf K}

\newcommand{\boZ}{\mathbb Z}
\newcommand{\boQ}{\mathbb Q}
\newcommand{\boL}{\mathbb L}
\newcommand{\boR}{\mathbb R}
\newcommand{\boT}{\mathbb T}
\newcommand{\boA}{\mathbb A}
\newcommand{\boP}{\mathbb P}
\newcommand{\boX}{\mathbb X}
\newcommand{\boY}{\mathbb Y}

\newcommand{\Section}[1]{\bigskip\section{#1}\medskip}
\theoremstyle{plain}
\newtheorem{thm}{Theorem}[section]

\newtheorem{lem}[thm]{Lemma}
\newtheorem{prop}[thm]{Proposition}
\newtheorem{cor}[thm]{Corollary}
\theoremstyle{definition}
\newtheorem{ex}[thm]{Example}

\newtheorem{rem}[thm]{Remark}
\newtheorem{rems}[thm]{Remarks}

\begin{document}

\title{Matlis category equivalences \\ for a ring epimorphism}

\author[S.~Bazzoni]{Silvana Bazzoni}

\address[Silvana Bazzoni]{%
Dipartimento di Matematica ``Tullio Levi-Civita'' \\
Universit\`a di Padova \\
Via Trieste 63, 35121 Padova (Italy)}

\email{bazzoni@math.unipd.it}

\author[L.~Positselski]{Leonid Positselski}

\address[Leonid Positselski]{%
Institute of Mathematics of the Czech Academy of Sciences \\
\v Zitn\'a~25, 115~67 Praha~1 (Czech Republic); and
\newline\indent
Laboratory of Algebra and Number Theory \\
Institute for Information Transmission Problems \\
Moscow 127051 (Russia)}

\email{positselski@yandex.ru}


\begin{abstract}
 Under mild assumptions, we construct the two Matlis additive category
equivalences for an associative ring epimorphism $u\:R\rarrow U$.
 Assuming that the ring epimorphism is homological of flat/projective
dimension~$1$, we discuss the abelian categories of $u$\+comodules and
$u$\+contramodules and construct the recollement of unbounded derived categories of $R$\+modules, $U$\+modules, and complexes of $R$\+modules
with $u$\+co/contramodule cohomology.
 Further assumptions allow to describe the third category in
the recollement as the unbounded derived category of the abelian
categories of $u$\+comodules and $u$\+contramodules.
 For commutative rings, we also prove that any homological epimorphism
of projective dimension~$1$ is flat.
 Injectivity of the map~$u$ is not required.
\end{abstract}

\maketitle

\tableofcontents

\section*{Introduction}
\medskip

 The aim of this paper is to develop the basics of the theory of
comodules and contramodules for an associative ring epimorphism in
the maximal natural generality, and for the purpose of future reference.
 Let us start this introduction with explaining what the words in
the paper's title mean.

 A \emph{ring epimorphism} $u\:R\rarrow U$ is a homomorphism of
associative rings such that for every pair of parallel ring
homomorphisms $f$, $g\:U\rightrightarrows V$ the equation $fu=gu$
implies $f=g$.
 Equivalently, a ring homomorphism~$u$ is an epimorphism if and only if
the two induced maps $u\ot\id$ and ${\id}\ot u\:U\rightrightarrows
U\ot_R U$ are equal to each other, if and only if one or both of
the maps $u\ot\id$ and ${\id}\ot u$ are isomorphisms, and if and only if
the multiplication map $U\ot_RU\rarrow U$ is an isomorphism.
 Further equivalent conditions for a ring map~$u$ to be an epimorphism
are that the functor of restriction of scalars $u_*\:U\modl\rarrow R\modl$
is fully faithful, or that the functor $u_*\:\modr U\rarrow\modr R$ is
fully faithful~\cite[Section~XI.1]{St}.
 In a ring epimorphism $R\rarrow U$, the ring $U$ is commutative whenever
the ring $R$~is.

 The history of what is known as Matlis category equivalences goes back
to the paper of Harrison~\cite{Harr}, where two equivalences between
certain full additive subcategories of the category of abelian groups
were constructed.
 The first equivalence was provided by the functor of tensor product
with the abelian group $\boQ/\boZ$, with the inverse functor
$\Hom_\boZ(\boQ/\boZ,{-})$.
 The second equivalence was given by the pair of functors
$\Tor_1^\boZ(\boQ/\boZ,{-})$ and $\Ext^1_\boZ(\boQ/\boZ,{-})$.

 In Matlis' memoir~\cite[Section~3]{Mat}, the setting was generalized
as follows.
 Let $R$ be a commutative domain, $Q$ be its field of quotients, and
$K=Q/R$ be the quotient $R$\+module.
 Then there are two equivalences between certain full additive
subcategories of the category of $R$\+modules.
 The first equivalence is provided by the functor of tensor product with
the $R$\+module $K$, and the inverse functor is $\Hom_R(K,{-})$.
 The second equivalence is given by the pair of functors
$\Tor_1^R(K,{-})$ and $\Ext^1_R(K,{-})$, which are mutually inverse in
restriction to the respective subcategories.
 Moreover, in the book~\cite{Mat2} Matlis extended the first one of his
two category equivalences to the setting with an arbitrary commutative
ring~$R$ and its total ring of quotients~$Q$.

 Let us mention two further generalizations of the Matlis category
equivalences in two different directions, which appeared in the two recent
papers~\cite{PMat,FN}.
 In the paper~\cite[Section~5]{PMat}, the two Matlis additive category
equivalences were constructed for a localization $S^{-1}R$ of a commutative
ring $R$ with respect to a multiplicative subset $S\subset R$.
 Injectivity of the map $R\rarrow S^{-1}R$ was not assumed, but
the assumption that the projective dimension of the $R$\+module $S^{-1}R$
does not exceed~$1$ was made.
 In the paper~\cite[Section~4]{FN}, the first Matlis category equivalence
was constructed for certain injective epimorphisms of noncommutative
rings $R\rarrow Q$, where $Q$ is the localization of $R$ with respect to
a one-sided Ore subset of regular elements.

 In this paper, we construct the first Matlis additive category equivalence
for any ring epimorphism $u\:R\rarrow U$ such that $\Tor_1^R(U,U)=0$, and
the second Matlis category equivalence for any~$u$ such that
$\Tor_1^R(U,U)=0=\Tor_2^R(U,U)$.
 Let us emphasize that \emph{neither} injectivity of~$u$, \emph{nor} any
condition on the projective or flat dimension of the $R$\+module $U$ is
required for these results.
 Commutativity of the rings $R$ and $U$ is not assumed, either.

 Furthermore, assuming that $U$ has projective dimension at most~$1$ as
a left $R$\+module and flat dimension at most~$1$ as a right $R$\+module,
we construct what was called the \emph{triangulated Matlis equivalence}
in~\cite{PMat}.
 However, unlike in~\cite{PMat}, we do not deduce the Matlis equivalences
between additive categories of modules from the triangulated equivalence,
but prove them separately.
 This allows to obtain the extra generality mentioned above.

 The key role is played by the full subcategories of what we call
\emph{$u$\+comodules} and \emph{$u$\+contramodules} in $R\modl$.
 The former is defined as the full subcategory of all left $R$\+modules
annihilated by the derived functor $\Tor^R_{0,1}(U,{-})$, while
the latter is the Geigle--Lenzing right $\Ext_R^{0,1}$\+perpendicular
subcategory to~$U$ in the category of left $R$\+modules.
 Under the assumptions of the flat/projective dimension of $U$ not
exceeding~$1$, these are abelian categories with exact inclusion
functors into $R\modl$.
 With the respective assumptions, we show that the $u$\+comodules form
a Grothendieck abelian category, while the abelian category of
$u$\+contramodules is locally presentable with a projective generator.
 We also discuss adjoint functors to the identity inclusions of
these full subcategories into the category of left $R$\+modules.

 The triangulated Matlis equivalence is an equivalence between
the (bounded or unbounded) derived category of complexes of $R$\+modules
with $u$\+comodule cohomology modules and the similar derived category
of complexes of $R$\+modules with $u$\+contramodule cohomology modules.
 The \emph{recollement} of triangulated Matlis equivalence identifies
both these triangulated categories with the Verdier quotient category
of the derived category $\sD^\star(R\modl)$ by the image of the fully
faithful functor of restriction of scalars $u_*\:\sD^\star(U\modl)
\rarrow\sD^\star(R\modl)$ for a homological ring epimorphism~$u$,
\begin{equation} \label{general-triangulated-matlis}
 \sD_{u\co}^\star(R\modl)\cong\sD^\star(R\modl)/u_*\sD^\star(U\modl)
 \cong\sD_{u\ctra}^\star(R\modl).
\end{equation}

 Under certain additional assumptions (which hold whenever, but not only
when, $u$~is injective) the exact embedding functors of the full
subcategories of $u$\+comodules and $u$\+contramodules, $R\modl_{u\co}
\rarrow R\modl$ and $R\modl_{u\ctra}\rarrow R\modl$, induce fully
faithful functors between the derived categories, identifying
the leftmost and the rightmost categories
in~\eqref{general-triangulated-matlis} with the derived categories of
the abelian categories $R\modl_{u\co}$ and $R\modl_{u\ctra}$.
 Hence one obtains an equivalence between the two derived categories,
\begin{equation} \label{particular-triangulated-matlis}
 \sD^\star(R\modl_{u\co})\cong\sD^\star(R\modl_{u\ctra}).
\end{equation}

 We should mention that, with the same assumptions as ours,
the equivalence of derived
categories~\eqref{particular-triangulated-matlis} was
obtained in~\cite[Corollary~4.4]{CX} as a particular case of a general
result about derived decomposition of abelian categories.
 The general approach in~\cite{CX} is based on the technique of complete
Ext-orthogonal pairs in abelian categories, which was introduced
by Krause and \v St'ov\'\i\v cek in~\cite{KS} (see also~\cite{BS}).
 The same argument as in the present paper, going back to~\cite{Pmgm}
and~\cite{PMat}, is used in~\cite{CX} in order to prove that
the triangulated functors induced by the embeddings of abelian
subcategories are fully faithful.
 One difference between our approaches is that in the present paper we
also obtain the equivalences~\eqref{general-triangulated-matlis}
holding under weaker assumptions.

 One of the main results of this paper is based on some recent results
of Hrbek and Angeleri H\"ugel--Hrbek~\cite{Hrb,AH}.
 We show that whenever $u\:R\rarrow U$ is a homological epimorphism of
commutative rings and $U$ is an $R$\+module of projective dimension~$1$,
it follows that $U$ is a flat $R$\+module.
 Generalizing Matlis' classical result, we also show that, under a mild
assumption on an epimorphism of commutative rings $u\:R\rarrow U$,
the ring $\R$ of endomorphisms of the complex $R\rarrow U$ in
the derived category of $R$\+modules is commutative.
 Under certain assumptions, it follows that the ring of endomorphisms
of the $R$\+module $U/R=\coker u$ is commutative, too.

 In the last section, we compute the full subcategories of $u$\+comodules
and $u$\+contramodules for certain ring epimorphisms originating
from the finite-dimensional noncommutative algebra $R$ associated with
the Kronecker quiver.
 We are grateful to Jan \v St\!'ov\'\i\v cek for the suggestion to
consider this example.

\medskip
\textbf{Acknowledgment.}
 The first-named author was partially supported by MIUR-PRIN
(Categories, Algebras: Ring-Theoretical and Homological
Approaches-CARTHA) and DOR1828909 of Padova University.
 The second-named author is supported by the GA\v CR project 20-13778S
and research plan RVO:~67985840. {\hbadness=1700\par}

\Section{First Additive Category Equivalence}
\label{first-additive-equivalence}

 Let $u\:R\rarrow U$ be an epimorphism of associative rings (i.~e.,
a ring homomorphism such that the multiplication map $U\ot_RU\rarrow U$
is an isomorphism of $R$\+$R$\+bimodules).
 Then one has $U\ot_RD\cong D\cong\Hom_R(U,D)$ for all left
$U$\+modules $D$, and the functor of restriction of scalars
$u_*\:U\modl\rarrow R\modl$ is fully faithful.
 The similar assertions hold for the right modules.
 We will say that a certain $R$\+module ``is a $U$\+module'' if it
belongs to the image of the functor of restriction of scalars.

 Let us introduce notation for the functors of extension and coextension
of scalars.
 The functor $u^*\:R\modl\rarrow U\modl$ left adjoint to~$u_*$
takes a left $R$\+module $M$ to the left $U$\+module $u^*(M)=U\ot_RM$.
 The functor $u^!\:R\modl\rarrow U\modl$ right adjoint to~$u_*$
takes a left $R$\+module $C$ to the left $U$\+module $u^!(C)=
\Hom_R(U,C)$.
 The natural isomorphisms of $U$\+modules mentioned in the previous
paragraph mean that the adjunction counit $u^*u_*\rarrow\Id$ and
the adjunction unit $\Id\rarrow u^!u_*$ are isomorphisms of endofunctors
$U\modl\rarrow U\modl$.

 We will use the simple notation $U/R$ for the cokernel of the map
$u\:R\rarrow U$.
 So $U/R$ is an $R$\+$R$\+bimodule.

 A left $R$\+module $M$ is called a \emph{$u$\+comodule} (or
a \emph{left $u$\+comodule}) if
$$
 U\otimes_RM=0=\Tor^R_1(U,M).
$$
 Similarly, a right $R$\+module $N$ is said to be a $u$\+comodule
(or a right $u$\+comodule) if $N\ot_RU=0=\Tor^R_1(N,U)$.

 A left $R$\+module $C$ is called a \emph{$u$\+contramodule} (or
a \emph{left $u$\+contramodule}) if
$$
 \Hom_R(U,C)=0=\Ext^1_R(U,C).
$$

 By~\cite[Proposition~1.1]{GL}, the class of all left $u$\+comodules
is closed under direct sums, cokernels of morphisms, and
extensions in $R\modl$.
 The class of all left $u$\+contramodules is closed under products,
kernels of morphisms, and extensions.

\begin{ex}
 The following example explains the ``comodules and contramodules''
terminology.
 Let $R=k[x]$ be the ring of polynomials in one variable over a field~$k$,
let $U=k[x,x^{-1}]$ be ring of Laurent polynomials, and let $u\:R\rarrow U$
be the natural inclusion.
 So one obtains the ring $U$ from $R$ by inverting the single element~$x$.

 Let $\C$ be the coalgebra over~$k$ such that the dual topological
algebra $\C^*$ is identified with the ring of formal power series $k[[x]]$.
 Then the full subcategory of $u$\+comodules in $R\modl$ is equivalent
to the category of comodules over the coalgebra $\C$, while the full
subcategory of $u$\+contramodules in $R\modl$ is equivalent to
the category of $\C$\+contramodules~\cite[Sections~1.3 and~1.6]{Prev}.
\end{ex}

 We will use the notation $\pd{}_RE$ for the projective dimension of
a left $R$\+module $E$ and $\fd E_R$ for the flat dimension of
a right $R$\+module~$E$.

 We will say that a left $R$\+module $A$ is \emph{$u$\+torsionfree} if
it is an $R$\+submodule of a left $U$\+module, or equivalently, if
the map $A\rarrow U\ot_RA$ induced by the ring homomorphism~$u$ is
injective.
 In other words, this means that the evaluation at $A$ of the adjunction
unit $\Id\rarrow u_*u^*$ is a monomorphism in $R\modl$.
 Similarly, we will say that a left $R$\+module $B$ is
\emph{$u$\+divisible} if it is a quotient module of a left
$U$\+module, or equivalently, if the map $\Hom_R(U,B)\rarrow B$
induced by~$u$ is surjective.
 In other words, this means that the evaluation at $B$ of the adjunction
counit $u_*u^!\rarrow\Id$ is an epimorphism in $R\modl$.

 Clearly, the class of all $u$\+torsionfree left $R$\+modules is
closed under subobjects, direct sums, and products in $R\modl$.
 Any left $R$\+module $A$ has a unique maximal $u$\+torsionfree
quotient module, which can be computed as the image of the natural
$R$\+module morphism $A\rarrow U\ot_RA$.
 The class of all $u$\+divisible left $R$\+modules is closed
under quotients, direct sums, and products.
 Any left $R$\+module $B$ has a unique maximal $u$\+divisible
submodule, which can be computed as the image of the natural
$R$\+module morphism $\Hom_R(U,B)\rarrow B$.

 A left $R$\+module $A$ is said to be \emph{$u$\+torsion} if its
maximal $u$\+torsionfree quotient module vanishes, or equivalently,
if $U\ot_RA=0$.
 Indeed, the $U$\+module $U\ot_RA$ is always generated by the image of
the map $u\ot\id_A\:A\rarrow U\ot_RA$; hence if the image of
$u\ot_R\id_A$ vanishes, then so does the whole module $U\ot_RA$.
 A left $R$\+module $B$ is said to be \emph{$u$\+reduced} if its
maximal $u$\+divisible submodule vanishes, or equivalently, if
$\Hom_R(U,B)=0$.
 Indeed, the map $\Hom(u,\id_B)\:\Hom_R(U,B)\rarrow B$ assigns
to an $R$\+module morphism $f\:U\rarrow B$ the element $f(1)\in B$.
 The action of $U$ in the left $R$\+module $\Hom_R(U,B)$ is given
by the rule $(vf)(w)=f(wv)$ for all $v$, $w\in U$.
 Hence if the image of the map $\Hom(u,\id_B)$ vanishes, then $f(v)=
(vf)(1)=\Hom(u,\id_B)(vf)=0$ for all $f\in\Hom_R(U,B)$ and $v\in U$,
so $f=0$.

\begin{rems}
 (1)~The commonly accepted terminology concerning divisibility goes back
to Matlis' memoir~\cite{Mat}, where the case of a commutative domain
$R$ with the field of fractions $Q$ was considered.
 In that context, an $R$\+module $B$ is said to be \emph{divisible} if
the action map $r\:B\rarrow B$ is surjective for every nonzero
element $r\in R$.
 An $R$\+module $B$ is said to be \emph{h\+divisible} if it is
a quotient module of a $Q$\+vector space.
 Similarly, an $R$\+module $B$ is said to be \emph{reduced} if it does
not have nonzero divisible submodules; and $B$ is \emph{h\+reduced} if
$\Hom_R(Q,B)=0$.
 Any h\+divisible $R$\+module is divisible, and any reduced $R$\+module
is h\+reduced; but the converse assertions do not hold in general.
 In fact, every divisible $R$\+module is h\+divisible if and only if
every h\+reduced $R$\+module is reduced and if and only if
$\pd{}_RQ\le1$ \cite[Theorem~10.1]{Mat}, \cite[Theorem~2.6]{Ham}
(domains $R$ satisfying these conditions are called \emph{Matlis
domains}).
 See~\cite[Lemma~1.8]{PMat}, \cite[Proposition~2.1(2)]{BP},
or~\cite[Theorem~6.3 and Example~6.5]{MS} together
with~\cite[Proposition~4.4]{AH} for generalizations.

 The classical definitions of divisible and reduced modules cannot be
extended to the setting in which the localization morphism
$q\:R\rarrow Q$ is replaced by a noncommutative ring
epimorphism $u\:R\rarrow U$.
 Our definitions of $u$\+divisible and $u$\+reduced modules generalize
the classical h\+divisibility and h\+reducedness properties.

\smallskip
 (2)~Let us warn the reader that our terminology is slightly confusing:
a left $R$\+module with no nonzero $u$\+torsion submodules does
\emph{not} need to be $u$\+torsionfree (unless $\fd U_R\le1$, as
we will see below).
 Similarly, a left $R$\+module with no $u$\+reduced quotient
modules does not need to be $u$\+divisible (unless $\pd{}_RU\le1$).
 The latter phenomenon manifests itself already in the case of
a localization morphism $q\:R\rarrow Q$ as in~(1)
(see~\cite[Theorem~10.1]{Mat} or~\cite[Lemma~1.8 and
Theorem~1.9]{Mat2}).
 The problem is that, unless the mentioned homological dimension
conditions are imposed on the $R$\+$R$\+bimodule $U$ or the ring
homomorphism~$u$, the classes of $u$\+torsionfree and $u$\+divisible
left $R$\+modules do not need to be closed under extensions.
\end{rems}

 The following theorem provides what appears to be the maximal natural
generality for the first of the two classical \emph{Matlis category
equivalences}~\cite[Theorem~3.4]{Mat}, \cite[Corollary~2.4]{Mat2} (going
back to Harrison's~\cite[Proposition~2.1]{Harr}).

\begin{thm} \label{first-matlis}
 Assume that\/ $\Tor^R_1(U,U)=0$.
 Then the restrictions of the adjoint functors
$M\longmapsto\Hom_R(U/R,M)$ and $C\longmapsto (U/R)\ot_RC$ are mutually
inverse equivalences between the additive categories of
$u$\+divisible left $u$\+comodules $M$ and $u$\+torsionfree
left $u$\+contramodules~$C$.
\end{thm}

 Before proceeding to prove the theorem, let us formulate and prove
a lemma.

\begin{lem} \label{matlis-hom-tensor-lemma}
 If\/ $\Tor^R_1(U,U)=0$, then \par
\textup{(a)} for any left $R$\+module $M$, the left $R$\+module\/
$\Hom_R(U/R,M)$ is a $u$\+torsionfree $u$\+contramodule; \par
\textup{(b)} for any left $R$\+module $C$, the left $R$\+module\/
$(U/R)\ot_RC$ is a $u$\+divisible $u$\+comodule.
\end{lem}

\begin{proof}
 Part~(a): the left $R$\+module $\Hom_R(U/R,M)$ is $u$\+torsionfree as
an $R$\+submod\-ule of the left $U$\+module $\Hom_R(U,M)$.
 Furthermore, since $U\ot_RU=U$, we have $(U/R)\ot_RU=0$, and therefore
$\Hom_R(U,\Hom_R(U/R,M))=0$.

 To show that $\Ext^1_R(U,\Hom_R(U/R,M))=0$, one observes that our
assumptions $U\ot_RU=U$ and $\Tor^R_1(U,U)=0$ imply
$\Tor^R_1(U/R,U)=0$, because the map $(R/\ker(u))\ot_RU\rarrow U$ is
an isomorphism.

 For any associative rings $R$ and $S$, left $R$\+module $L$,
\ $S$\+$R$\+bimodule $E$, and left $S$\+module $M$ such that
$\Tor^R_1(E,L)=0$, there is a natural
injective map of abelian groups
$$
 \Ext^1_R(L,\Hom_S(E,M))\lrarrow\Ext^1_S(E\ot_RL,\>M).
$$
 In particular, in the situation at hand $\Ext^1_R(U,\Hom_R(U/R,M))$
is a subgroup of $\Ext^1_R((U/R)\ot_RU,\>M)=0$.

 The proof of part~(b) is dual-analogous.
 The left $R$\+module $(U/R)\ot_RC$ is $u$\+divisible as
a quotient $R$\+module of the left $U$\+module $U\ot_RC$.
 Since $U\ot_R(U/R)=0$, we have $U\ot_R(U/R)\ot_RC=0$.

 For any associative rings $R$ and $S$, right $R$\+module $B$, \
$R$\+$S$\+bimodule $E$, and left $S$\+module $C$ such that
$\Tor^R_1(B,E)=0$, there is a natural surjective map of abelian groups
$$
 \Tor^S_1(B\ot_RE,\>C)\lrarrow\Tor^R_1(B,\>E\ot_SC).
$$
 In particular, in the situation at hand $\Tor^R_1(U,\>(U/R)\ot_RC)$ is
a quotient group of $\Tor_1^R(U\ot_R(U/R),\>C)=0$.
 For a more high-tech derived category/spectral sequence presentation
of the same argument, see
Lemmas~\ref{matlis-tor-1-lemma}\+-\ref{matlis-ext-1-lemma} below.
\end{proof}

\begin{proof}[Proof of Theorem~\ref{first-matlis}]
 By Lemma~\ref{matlis-hom-tensor-lemma}, the functors $M\longmapsto
\Hom_R(U/R,M)$ and $C\longmapsto(U/R)\ot_RC$ take $u$\+divisible
left $u$\+comodules to $u$\+torsionfree left $u$\+contramodules and back
(in fact, they take arbitrary left $R$\+modules to left $R$\+modules
from these two classes).
 It remains to show that the restrictions of these functors to these
two full subcategories in $R\modl$ are mutually inverse equivalences
between them. {\hbadness=1900\par}

 Let $M$ be a $u$\+divisible left $u$\+comodule.
 We will show that the adjunction morphism $(U/R)\ot_R\Hom_R(U/R,M)
\rarrow M$ is an isomorphism.
 Since $M$ is $u$\+divisible, we have a natural short exact sequence
of left $R$\+modules
\begin{equation} \label{divisible-module-sequence}
 0\lrarrow\Hom_R(U/R,M)\lrarrow\Hom_R(U,M)\lrarrow M\lrarrow0.
\end{equation}
 Since the left $R$\+module $\Hom_R(U/R,M)$ is $u$\+torsionfree,
we also have a natural short exact sequence of left $R$\+modules
\begin{equation}\label{hom-module-torsionfree-sequence}
 0\rarrow\Hom_R(U/R,M)\rarrow U\ot_R\Hom_R(U/R,M)\rarrow
 U/R\ot_R\Hom_R(U/R,M)\rarrow0.
\end{equation}
 Since $M$ is a $u$\+comodule, applying the functor $U\ot_R{-}$ to
the short exact sequence~\eqref{divisible-module-sequence} produces
an isomorphism $U\ot_R\Hom_R(U/R,M)\cong U\ot_R\Hom_R(U,M)=\Hom_R(U,M)$.
 Now the commutative diagram
$$
\begin{tikzcd}
 \Hom_R(U/R,M) \arrow[r] \arrow[dd, Leftrightarrow, no head, no tail]
 & U\ot_R\Hom_R(U/R,M) \arrow[r] \arrow[d, "\cong"]
 & U/R\ot_R\Hom_R(U/R,M) \arrow[dd] \\
 & U\ot_R\Hom_R(U,M) \arrow[d, "\cong"] \\
 \Hom_R(U/R,M)\arrow[r] & \Hom_R(U,M) \arrow[r] & M
\end{tikzcd}
$$
shows that we have a morphism from the short exact
sequence~\eqref{hom-module-torsionfree-sequence} to the short exact sequence~\eqref{divisible-module-sequence} that is the identity on
the leftmost terms, an isomorphism on the middle terms, and
the adjunction morphism on the rightmost terms.
 Therefore, the adjunction morphism is an isomorphism.

 Let $C$ be a $u$\+torsionfree left $u$\+contramodule.
 Let us show that the adjunction morphism $C\rarrow\Hom_R(U/R,\>
(U/R)\ot_RC)$ is an isomorphism.
 Since $C$ is $u$\+torsionfree, we have a natural short exact
sequence of left $R$\+modules
\begin{equation} \label{torsionfree-module-sequence}
 0\lrarrow C\lrarrow U\ot_RC\lrarrow (U/R)\ot_RC\lrarrow0.
\end{equation}
 Since the left $R$\+module $(U/R)\ot_RC$ is $u$\+divisible, we also
have a natural short exact sequence of left $R$\+modules
\begin{equation} \label{tensor-module-divisible-sequence}
 0\rarrow\Hom_R(U/R,\>(U/R)\ot_RC)\rarrow\Hom_R(U,\>(U/R)\ot_RC)
 \rarrow (U/R)\ot_RC\rarrow0.
\end{equation}
 Since $C$ is a $u$\+contramodule, applying the functor $\Hom_R(U,{-})$
to the short exact sequence~\eqref{torsionfree-module-sequence}
produces an isomorphism $U\ot_RC=\Hom_R(U,\>U\ot_RC)\cong
\Hom_R(U,\>\allowbreak(U/R)\ot_RC)$.
 Now the commutative diagram
$$
\begin{tikzcd}
 C \arrow[r] \arrow[dd] & U\ot_R C \arrow[r] \arrow[d, "\cong"] &
 (U/R)\ot_RC \arrow[dd, Leftrightarrow, no head, no tail] \\
 & \Hom_R(U,\>U\ot_RC) \arrow[d, "\cong"] \\
 \Hom_R(U/R,\>(U/R)\ot_RC) \arrow[r] & \Hom_R(U,\>(U/R)\ot_RC)
 \arrow[r] & (U/R)\ot_RC
\end{tikzcd}
$$
shows that we have a morphism from the short exact
sequence~\eqref{torsionfree-module-sequence} to the short exact
sequence~\eqref{tensor-module-divisible-sequence} that is the identity
on the rightmost terms, an isomorphism on the middle terms, and
the adjunction morphism on the leftmost terms.
 Therefore, the adjunction morphism is an isomorphism.
\end{proof}

\Section{Second Additive Category Equivalence}
\label{second-additive-equivalence}

 Let $K^\bu$ denote the two-term complex $R\rarrow U$, with the term
$R$ placed in the cohomological degree~$-1$ and the term $U$ in
the cohomological degree~$0$.
 We will view $K^\bu$ as an object of the bounded derived category of
$R$\+$R$\+bimodules $\sD^\bb(R\bimod R)$.
 So, there is a distinguished triangle
\begin{equation} \label{main-triangle}
 R\lrarrow U\lrarrow K^\bu\lrarrow R[1]
\end{equation}
in the triangulated category $\sD^\bb(R\bimod R)$.

 Alternatively, the complex $K^\bu$ can be considered as an object
of the bounded derived category of left $R$\+modules $\sD^\bb(R\modl)$
endowed with a right action of the ring $R$ by its derived category
object endomorphisms, or as an object of the bounded derived category
of right $R$\+modules $\sD^\bb(\modr R)$ endowed with a left action
of~$R$.
 Then~\eqref{main-triangle} is viewed as a distinguished triangle
in $\sD^\bb(R\modl)$ or $\sD^\bb(\modr R)$.

 By an abuse of notation, given a left $R$\+module $B$, we will denote
simply by 
$$
 \Ext^n_R(K^\bu,B)=H^n(\boR\Hom_R(K^\bu,B))=
 \Hom_{\sD^\bb(R\modl)}(K^\bu,B[n])
$$
the abelian group of all morphisms $K^\bu\rarrow B[n]$ in the derived
category of left $R$\+modules.
 The right action of $R$ in the object $K^\bu\in\sD^\bb(R\modl)$ induces
a left $R$\+module structure on the groups $\Ext^n_R(K^\bu,B)$.
 
 Similarly, we set
$$
 \Tor_n^R(K^\bu,A)=H^{-n}(K^\bu\ot_R^\boL A)
$$
for any left $R$\+module~$A$.
 Here $K^\bu$ is viewed as an object of the bounded derived category of
right $R$\+modules for the purpose of computing the derived tensor
product $K^\bu\ot_R^\boL A$, and then the left action of $R$ in
the object $K^\bu\in\sD^\bb(\modr R)$ induces a left $R$\+module
structure on the groups $\Tor^R_n(K^\bu,A)$.

\begin{lem} \label{tor-ext-with-K}
 For every left $R$\+module $C$, there are natural isomorphisms
of left $R$\+modules \par
\textup{(a)} $\Tor_n^R(K^\bu,C)=0=\Ext^n_R(K^\bu,C)$ for $n<0$; \par
\textup{(b)} $\Tor_n^R(K^\bu,C)=\Tor_n^R(U,C)$ and\/
$\Ext^n_R(K^\bu,C)=\Ext^n_R(U,C)$ for all $n>1$; \par
\textup{(c)} $\Tor_0^R(K^\bu,C)=(U/R)\ot_R C$ and\/
$\Ext^0_R(K^\bu,C)=\Hom_R(U/R,C)$.
\end{lem}

\begin{proof}
 All the assertions follow immediately from the (co)homology long exact
sequences obtained by applying the functors $\boR\Hom_R({-},C)$ and
${-}\ot_R^\boL C$ to the distinguished triangle~\eqref{main-triangle}.
\end{proof}

 Furthermore, for any left $R$\+modules $A$ and $B$ there are five-term
exact sequences of low-dimensional $\Tor$ and $\Ext$ induced by
the distinguished triangle~\eqref{main-triangle}:
\begin{multline} \label{main-tor-sequence}
 0\lrarrow\Tor^R_1(U,A)\lrarrow\Tor^R_1(K^\bu,A) \\
 \lrarrow A\lrarrow U\ot_R A\lrarrow\Tor_0^R(K^\bu,A)\lrarrow0
\end{multline}
and
\begin{multline} \label{main-ext-sequence}
 0\lrarrow\Ext^0_R(K^\bu,B)\lrarrow\Hom_R(U,B)\lrarrow B \\
 \lrarrow\Ext^1_R(K^\bu,B)\lrarrow\Ext^1_R(U,B)\lrarrow0.
\end{multline}
 Both~\eqref{main-tor-sequence} and~\eqref{main-ext-sequence} are exact
sequences of left $R$\+modules.

 Borrowing the terminology of Matlis~\cite{Mat},
 we will say that a left $R$\+module $A$ is \emph{$u$\+special} if
the map $A\rarrow U\ot_RA$ is surjective.
 Equivalently (in view of the exact sequence~\eqref{main-tor-sequence}
or Lemma~\ref{tor-ext-with-K}(c)), this means that $\Tor^R_0(K^\bu,A)=0$.
 Similarly, a left $R$\+module $B$ is \emph{$u$\+cospecial} if the map
$\Hom_R(U,B)\rarrow B$ is injective.
 Equivalently (by the exact sequence~\eqref{main-ext-sequence}
or Lemma~\ref{tor-ext-with-K}(c)), this means that
$\Ext_R^0(K^\bu,B)=0$.

 The next lemma provides another characterization of
$u$\+special and $u$\+cospecial modules.

\begin{lem} \label{special-and-cospecial}
\textup{(a)} A left $R$\+module $A$ is $u$\+special if and only if
its maximal $u$\+torsionfree quotient module is a $U$\+module. \par
\textup{(b)} A left $R$\+module $B$ is $u$\+cospecial if and only if
its maximal $u$\+divisible submodule is a $U$\+module.
\end{lem}

\begin{proof}
 Part~(b): if $B$ is $u$\+cospecial, then the morphism $\Hom_R(U,B)
\rarrow B$ is injective, so $\Hom_R(U,B)$ is the maximal
$u$\+divisible submodule of~$B$.
 Conversely, if the maximal $u$\+divisible submodule of $B$ is
a $U$\+module $D$, then $\Hom_R(U/R,B)=\Hom_R(U/R,D)=0$.

 Part~(a): if $A$ is $u$\+special, then the morphism $A\rarrow U\ot_RA$
is surjective, so $U\ot_RA$ is the maximal $u$\+torsionfree
quotient module of~$A$.
 Conversely, assume that the maximal $u$\+torsionfree quotient module
of $A$ is a $U$\+module~$D$.
 Note that $\Hom_\boZ(U/R,\boQ/\boZ)$ is a $u$\+torsionfree left
$R$\+module, because it is a submodule of the left $U$\+module
$\Hom_\boZ(U,\boQ/\boZ)$.
 Hence
\begin{multline*}
 \Hom_\boZ((U/R)\ot_RA,\>\boQ/\boZ)\cong
 \Hom_R(A,\Hom_\boZ(U/R,\boQ/\boZ)) \\
 =\Hom_R(D,\Hom_\boZ(U/R,\boQ/\boZ))\cong
 \Hom_\boZ((U/R)\ot_R D,\>\boQ/\boZ),
\end{multline*}
and the last term is zero since $D$ is a $U$\+module.
 It follows that $(U/R)\ot_RA=0$.
\end{proof}

 The following theorem is our version of the second Matlis
category equivalence~\cite[Theorem~3.8]{Mat} (going back to
Harrison's~\cite[Proposition~2.3]{Harr}).

\begin{thm} \label{second-matlis}
 Assume that\/ $\Tor^R_1(U,U)=0=\Tor^R_2(U,U)$.
 Then the restrictions of the functors
$M\longmapsto\Ext_R^1(K^\bu,M)$ and $C\longmapsto\Tor^R_1(K^\bu,C)$ are
mutually inverse equivalences between the additive categories of
$u$\+cospecial left $u$\+comodules $M$ and $u$\+special
left $u$\+contramodules~$C$.
\end{thm}

 We are going to use a fairly well-known spectral sequence technique.
 It is formulated in the proposition below, for lack of a suitable
reference covering the required generality.
 The following generalization of the notation introduced in
the beginning of this section is presumed in the proposition.

 Given an associative ring $S$, a complex of left $S$\+modules $L^\bu$,
and a complex of right $S$\+modules $M^\bu$, we put
$\Tor_n^S(M^\bu,L^\bu)=H^{-n}(M^\bu\ot_S^\boL L^\bu)$ for every
$n\in\boZ$.
 If $M^\bu$ is a complex of $R$\+$S$\+bimodules, then
$M^\bu\ot_S^\boL L^\bu$ is an object of the derived category
$\sD(R\modl)$ and the abelian groups $\Tor_n^S(M^\bu,L^\bu)$ have
left $R$\+module structures.
 Given two complexes of left $S$\+modules $M^\bu$ and $L^\bu$,
we put $\Ext_S^n(M^\bu,L^\bu)=H^n(\boR\Hom_S(M^\bu,L^\bu))
=\Hom_{\sD(S\modl)}(M^\bu,L^\bu[n])$ for every $n\in\boZ$.
 If $M^\bu$ is a complex of $S$\+$R$\+bimodules, then
$\boR\Hom_S(M^\bu,L^\bu)$ is an object of the derived category
$\sD(R\modl)$ and the abelian groups $\Ext_S^n(M^\bu,L^\bu)$
have left $R$\+module structures.

 In order to avoid spectral sequence convergence issues, we assume our
complexes to be suitably bounded.

\begin{prop} \label{associativity-spectral-sequences}
\textup{(a)} Let $L^\bu$ be a bounded above complex of left
$S$\+modules, $M^\bu$ be a bounded above complex of $R$\+$S$\+bimodules,
and $N^\bu$ be a bounded above complex of right $R$\+modules.
 Then there is a spectral sequence of abelian groups
$$
 E^2_{pq}=\Tor_p^R(N^\bu,\Tor_q^S(M^\bu,L^\bu))\Longrightarrow
 E^\infty_{pq}=\mathrm{gr}_p
 H^{-p-q}(N^\bu\ot_R^\boL M^\bu\ot_S^\boL L^\bu),
$$
with the differentials $\partial^r_{pq}\:E^r_{pq}\rarrow E^r_{p-r,q+r-1}$,
\ $r\ge2$, converging to the Tor groups
$H^{-p-q}(N^\bu\ot_R^\boL M^\bu\ot_S^\boL L^\bu)=
\Tor^S_{p+q}(N^\bu\ot_R^\boL M^\bu,\>L^\bu)$. \par
\textup{(b)} Let $L^\bu$ be a bounded below complex of left
$S$\+modules, $M^\bu$ be a bounded above complex of $S$\+$R$\+bimodules,
and $N^\bu$ be a bounded above complex of left $R$\+modules.
 Then there is a spectral sequence of abelian groups
$$
 E_2^{pq}=\Ext^p_R(N^\bu,\Ext^q_S(M^\bu,L^\bu))\Longrightarrow
 E_\infty^{pq}=\mathrm{gr}^p
 H^{p+q}(\boR\Hom_S(M^\bu\ot^\boL_R N^\bu,\>L^\bu)),
$$
with the differentials $d_r^{pq}\:E_r^{pq}\rarrow E_r^{p+r,q-r+1}$, \
$r\ge2$, converging to the Ext groups
$H^{p+q}(\boR\Hom_R(N^\bu,\,\boR\Hom_S(M^\bu,L^\bu)))=
H^{p+q}(\boR\Hom_S(M^\bu\ot^\boL_R N^\bu,\>L^\bu))=
\Ext_S^{p+q}(M^\bu\ot_R^\boL N^\bu,\> L^\bu)$.
\end{prop}

\begin{proof}
 Let us briefly explain part~(a), which is a straightforward
generalization of~\cite[Exercise~5.6.2]{Wei}.
 Replace the complex $L^\bu$ by a quasi-isomorphic complex of
flat left $S$\+modules $Q^\bu$ and the complex $N^\bu$ by
a quasi-isomorphic complex of flat right $R$\+modules~$P^\bu$
(where both the complexes $P^\bu$ and $Q^\bu$ are bounded above).
 Denote by $D^\bu$ the total complex of the bicomplex of
left $R$\+modules $M^\bu\ot_S Q^\bu$.
 Then one has $\Tor_q^S(M^\bu,L^\bu)=H^{-q}(D^\bu)$.
 Consider the bicomplex $C^{\bu,\bu}=P^\bu\ot_RD^\bu$ with the terms
$C^{-p,-q}=P^{-p}\ot_RD^{-q}$.
 Then $H^{-q}(C^{-p,\bu})\cong P^{-p}\ot_R\Tor_q^S(M^\bu,L^\bu)$
and consequently $H^{-p}(H^{-q}(C^{\bu,\bu}))\cong
\Tor_p^R(N^\bu,\Tor^S_q(M^\bu,L^\bu))$ for every~$p$ and~$q$.

 On the other hand, the total complex $C^\bu=
P^\bu\ot_R M^\bu\ot_S Q^\bu$ of the bicomplex $C^{\bu,\bu}$
represents the object $N^\bu\ot_R^\boL M^\bu\ot_S^\boL L^\bu$
in the derived category of abelian groups.
 Thus the general construction of the spectral sequence of a double
complex (or rather, the appropriate one of two such spectral
sequences~\cite[Definition~5.6.1]{Wei}) applied to the bicomplex
$C^{\bu,\bu}$ provides part~(a).
 Part~(b) is similar.
\end{proof}

 Before proceeding to prove Theorem~\ref{second-matlis}, we formulate
two lemmas, which extend the result of
Lemma~\ref{matlis-hom-tensor-lemma}.

\begin{lem} \label{matlis-tor-1-lemma}
\textup{(a)} If\/ $\Tor^R_1(U,U)=0$, then the left $R$\+module\/
$\Tor^R_0(K^\bu,A)$ is a $u$\+comodule for any left $R$\+module~$A$.
\par
\textup{(b)} If\/ $\Tor^R_1(U,U)=0=\Tor_2^R(U,U)=0$, then
the left $R$\+module\/ $\Tor^R_1(K^\bu,A)$ is a $u$\+comodule for any
left $R$\+module $A$ such that\/ $\Tor^R_0(K^\bu,A)=0$. \par
\textup{(c)} If\/ $\Tor^R_1(U,U)=0$ and\/ $\fd U_R\le1$, then the left
$R$\+module\/ $\Tor^R_1(K^\bu,A)$ is a $u$\+comodule for any left
$R$\+module~$A$.
\end{lem}

\begin{proof}
 Following Proposition~\ref{associativity-spectral-sequences}(a),
there is a spectral sequence
$$
 E^2_{pq}=\Tor_p^R(U,\Tor_q^R(K^\bu,A))\Longrightarrow
 E^\infty_{pq}=\mathrm{gr}_p\Tor_{p+q}^R(U\ot_R^\boL K^\bu,\>A).
$$
 Clearly, one has $\Tor_n^R(U\ot_R^\boL K^\bu,\>A)=0$ whenever
$H^{-i}(U\ot_R^\boL K^\bu)=0$ for all $0\le i\le n$.
 Since $U\ot_RU=U$, the latter condition holds whenever
$\Tor^R_i(U,U)=0$ for all $1\le i\le n$.
 Thus $E^\infty_{pq}=0$ whenever $p+q\le1$ in the assumptions of
part~(a), whenever $p+q\le2$ in the assumptions of part~(b),
and for all $p$, $q\in\boZ$ in the assumptions of part~(c).

 The differentials are $\partial^r_{pq}\:E^r_{pq}\rarrow
E^r_{p-r,q+r-1}$, \ $r\ge2$.
 Now all the differentials involving $E^r_{0,0}$ and $E^r_{1,0}$ vanish
for the dimension reasons, so $E^\infty_{0,0}=0=E^\infty_{1,0}$ implies
$E^2_{0,0}=0=E^2_{1,0}$.
 This proves part~(a).
 Furthermore, the only possibly nontrivial differentials involving
$E^r_{0,1}$ and $E^r_{1,1}$ are
$$
 \partial^2_{2,0}\:E^2_{2,0}\lrarrow E^2_{0,1}
 \quad\text{and}\quad
 \partial^2_{3,0}\:E^2_{3,0}\lrarrow E^2_{1,1}.
$$
 When $\Tor_0^R(K^\bu,A)=0$, one has $E^2_{p,0}=0$ for all $p\in\boZ$.
 When $\fd U_R\le1$, one has $E^2_{pq}=0$ for $p\ge2$ and all~$q$.
 In both cases, $E^\infty_{0,1}=0=E^\infty_{1,1}$
implies $E^2_{0,1}=0=E^2_{1,1}$, proving parts~(b) and~(c).
\end{proof}

\begin{lem} \label{matlis-ext-1-lemma}
\textup{(a)} If\/ $\Tor^R_1(U,U)=0$, then the left $R$\+module\/
$\Ext_R^0(K^\bu,B)$ is a $u$\+contramodule for any left $R$\+module~$B$.
\par
\textup{(b)} If\/ $\Tor^R_1(U,U)=0=\Tor_2^R(U,U)=0$, then
the left $R$\+module\/ $\Ext_R^1(K^\bu,B)$ is a $u$\+contramodule for
any left $R$\+module $B$ such that\/ $\Ext_R^0(K^\bu,B)=0$. \par
\textup{(c)} If\/ $\Tor^R_1(U,U)=0$ and\/ $\pd{}_RU\le1$, then the left
$R$\+module\/ $\Ext_R^1(K^\bu,B)$ is a $u$\+contramodule for any left
$R$\+module~$B$.
\end{lem}

\begin{proof}
 Dual-analogous to Lemma~\ref{matlis-tor-1-lemma}
(and similar to~\cite[Lemma~1.7]{PMat}).
 The spectral sequence of
Proposition~\ref{associativity-spectral-sequences}(b) is
a suitable tool.
\end{proof}

\begin{proof}[Proof of Theorem~\ref{second-matlis}]
 Let $M$ be a $u$\+cospecial left $u$\+comodule.
 By Lemma~\ref{matlis-ext-1-lemma}(b), the left $R$\+module
$\Ext^1_R(K^\bu,M)$ is a $u$\+contramodule.
 Furthermore, the exact sequence~\eqref{main-ext-sequence} for
the $R$\+module $M$ reduces to a four-term sequence
$$
 0\lrarrow\Hom_R(U,M)\lrarrow M\lrarrow\Ext^1_R(K^\bu,M)
 \lrarrow\Ext^1_R(U,M)\lrarrow0.
$$
 Denoting by $E$ the image of the map $M\rarrow\Ext^1_R(K^\bu,M)$, we
have two short exact sequences of left $R$\+modules
$0\rarrow\Hom_R(U,M)\rarrow M\rarrow E\rarrow 0$ and
$0\rarrow E\rarrow\Ext^1_R(K^\bu,M)\rarrow\Ext^1_R(U,M)\rarrow0$.

 The assumptions that $U\ot_RU=U$ and $\Tor^R_i(U,U)=0$ for $i=1$
and~$2$ imply that $U\ot_RD=D$ and $\Tor^R_i(U,D)=0$ for all
left $U$\+modules $D$ and $i=1$,~$2$.
 Hence (by Lemma~\ref{tor-ext-with-K} and the exact
sequence~\eqref{main-tor-sequence} for the $R$\+module $D$)
we have $\Tor^R_i(K^\bu,D)=0$ for $-1\le i\le 2$.
 In particular, this applies to the left $U$\+modules $D=\Hom_R(U,M)$
and $\Ext^1_R(U,M)$.

 Now from the long exact sequences of $\Tor^R_*(K^\bu,{-})$ related to
our two short exact sequences of left $R$\+modules we see that both
the maps $\Tor^R_i(K^\bu,M)\rarrow\Tor^R_i(K^\bu,E)\rarrow
\Tor^R_i(K^\bu,\Ext^1_R(K^\bu,M))$ are isomorphisms for $i=0$ and~$1$.
 In particular, $\Tor^R_0(K^\bu,\Ext^1_R(K^\bu,M))\cong\Tor^R_0(K^\bu,M)
=(U/R)\ot_RM=0$, since $U\ot_RM=0$.
 Hence the left $R$\+module $\Ext^1_R(K^\bu,M)$ is $u$\+special.

 Furthermore, the map $\Tor^R_1(K^\bu,M)\rarrow M$ in the short
exact sequence~\eqref{main-tor-sequence} is an isomorphism, since
$U\ot_RM=0=\Tor_1^R(U,M)$.
 Thus we obtain a natural isomorphism
$\Tor^R_1(K^\bu,\Ext^1_R(K^\bu,M))\cong M$.

 The dual-analogous argument shows that the left $R$\+module
$\Tor_1^R(K^\bu,C)$ is a $u$\+cospecial $u$\+comodule for any
$u$\+special $u$\+contramodule $C$, and provides a natural isomorphism
$\Ext_R^1(K^\bu,\Tor_1^R(K^\bu,C))\cong C$.
 One has to observe that $\Hom_R(U,D)=D$ and $\Ext_R^i(U,D)=0$
for all left $U$\+modules $D$ and $i=1$,~$2$, hence
$\Ext_R^i(K^\bu,D)=0$ for $-1\le i\le 2$, etc.
\end{proof}

 In the rest of this section we discuss how our theory simplifies and
improves with the assumptions that the projective dimension of the left
$R$\+module $U$ and/or the flat dimension of the right $R$\+module $U$
do not exceed~$1$.

\begin{lem} \label{torsionfree-divisible}
\textup{(a)} Assume that\/ $\Tor^R_1(U,U)=0$ and\/ $\fd U_R\le1$.
 Then a left $R$\+module $A$ is $u$\+torsionfree if and only if\/
$\Tor^R_1(K^\bu,A)=0$. \par
\textup{(b)} Assume that\/ $\Tor^R_1(U,U)=0$ and\/ $\pd{}_RU\le1$.
 Then a left $R$\+module $B$ is $u$\+divisible if and only if\/
$\Ext_R^1(K^\bu,B)=0$.
\end{lem}

\begin{proof}
 This is similar to~\cite[Lemma~5.1(b)]{PMat}.
 Let us prove part~(a).
 The ``if'' claim follows immediately from the exact
sequence~\eqref{main-tor-sequence}.
 To prove the ``only if'', assume that $A$ is $u$\+torsionfree.
 Then the exact sequence~\eqref{main-tor-sequence} implies
that the left $R$\+module morphism
$\Tor_1^R(U,A)\rarrow\Tor_1^R(K^\bu,A)$ is an isomorphism.
 Since $\Tor_1^R(K^\bu,A)$ is a left $u$\+comodule
by Lemma~\ref{matlis-tor-1-lemma}(c) and $\Tor_1^R(U,A)$ is
a left $U$\+module, they can only be isomorphic when
both of them vanish.
\end{proof}

 It is clear from the definition and
Lemma~\ref{torsionfree-divisible}(a) that, when $\Tor^R_1(U,U)=0$ and
$\fd U_R\le1$, the full subcategory of $u$\+torsionfree $R$\+modules
is closed under extensions, subobjects, direct sums, and products.
 So $u$\+torsionfree $R$\+modules form the torsionfree class of
a certain torsion pair in $R\modl$.
 The related torsion class is the class of all $u$\+torsion
$R$\+modules, that is, all left $R$\+modules $A$ such that $U\ot_RA=0$.

 Similarly, it is clear from the definition and
Lemma~\ref{torsionfree-divisible}(b) that, whenever $\Tor^R_1(U,U)=0$
and $\pd{}_RU\le1$, the full subcategory of $u$\+divisible
$R$\+modules is closed under extensions, quotients, direct sums
and products.
 So $u$\+divisible $R$\+modules form the torsion class of a certain
torsion theory in $R\modl$.
 The related torsionfree class is the class of all $u$\+reduced
$R$\+modules, that is, all left $R$\+modules $B$ such that
$\Hom_R(U,B)=0$.

 It is clear from the definition that the full subcategory of
$u$\+special left $R$\+modules is closed under extensions, quotients,
and direct sums.
 Hence it is the torsion class of a torsion pair in $R\modl$.
 When $\Tor^R_1(U,U)=0$ and $\fd U_R\le1$, the related torsionfree class
can be described as the class of all $u$\+torsionfree $u$\+reduced
left $R$\+modules.

 Similarly, the full subcategory of $u$\+cospecial left $R$\+modules is
closed under extensions, subobjects, direct sums, and products.
 Hence it is the torsionfree class of a torsion pair in $R\modl$.
 When $\Tor^R_1(U,U)=0$ and $\pd{}_RU\le1$, the related torsion class
can be described as the class of all $u$\+divisible $u$\+torsion
left $R$\+modules.

\Section{Abelian Categories of $u$-Comodules and $u$-Contramodules}
\label{u-co-and-contra-categories-secn}

 In this section, as in the previous one, $u\:R\rarrow U$ is
an associative ring epimorphism.
 For most of the results, we will have to assume that $u$~is
a \emph{homological} ring epimorphism, that is, $\Tor^R_i(U,U)=0$
for $i\ge1$.
 In fact, we will mostly have to assume either that the flat dimension
of the right $R$\+module $U$ does not exceed~$1$ (when discussing
left $u$\+comodules), or that the projective dimension of the left
$R$\+module $U$ does not exceed~$1$ (when considering left
$u$\+contramodules).

 Let us denote the full subcategory of left $u$\+comodules 
by $R\modl_{u\co}\subset R\modl$, and the full subcategory of
left $u$\+contramodules by $R\modl_{u\ctra}\subset R\modl$.
 For any left $R$\+module $C$, we set $\Gamma_u(C)=\Tor_1^R(K^\bu,C)$
and $\Delta_u(C)=\Ext^1_R(K^\bu,C)$.
 The natural left $R$\+module morphisms (occurring in the exact
sequences~(\ref{main-tor-sequence}\+-\ref{main-ext-sequence}))
are denoted by $\gamma_{u,C}\:\Gamma_u(C)\rarrow C$ and
$\delta_{u,C}\:C\rarrow\Delta_u(C)$.

\begin{prop} \label{u-comodule-category}
 Assume that\/ $\fd U_R\le 1$.  Then \par
\textup{(a)} the full subcategory $R\modl_{u\co}$ is closed under
the kernels, cokernels, extensions, and direct sums in $R\modl$.
 So $R\modl_{u\co}$ is an abelian category and the embedding functor
$R\modl_{u\co}\rarrow R\modl$ is exact; \par
\textup{(b)} assuming also that\/ $\Tor_1^R(U,U)=0$, the functor\/
$\Gamma_u\:R\modl\rarrow R\modl_{u\co}$ is right adjoint to the fully
faithful embedding functor $R\modl_{u\co}\rarrow R\modl$.
\end{prop}

\begin{proof}
 Part~(a) is a particular case of~\cite[Proposition~1.1]{GL}
or~\cite[Theorem~1.2(b)]{Pcta}.
 To prove part~(b), notice that $\Gamma_u(A)\in
R\modl_{u\co}$ for any $A\in R\modl$ by
Lemma~\ref{matlis-tor-1-lemma}(c).
 We have to show that for every left $R$\+module $A$, every left
$u$\+comodule $M$, and an $R$\+module morphism $M\rarrow A$ there exists
a unique $R$\+module morphism $M\rarrow\Gamma_u(A)$ making the triangle
diagram $M\rarrow\Gamma_u(A)\rarrow A$ commutative.

 Indeed, looking on the exact sequence~\eqref{main-tor-sequence},
the composition $M\rarrow A\rarrow U\ot_RA$ vanishes, since $U\ot_RM=0$.
 Now the obstruction to lifting the morphism $M\rarrow A$ to
a morphism $M\rarrow\Tor_1^R(K^\bu,A)$ lies in the group
$\Ext^1_R(M,\Tor_1^R(U,A))$, and the obstruction to uniqueness
of such a lifting lies in the group $\Hom_R(M,\Tor_1^R(U,A))$.
 
 Generally, for any ring homomorphism $R\rarrow U$, left $R$\+module
$M$, left $U$\+module $D$, right $U$\+module $E$, and an integer
$n\ge0$ such that $\Tor^R_i(U,M)=0$ for $1\le i\le n$, one has
$\Tor^R_n(E,M)\cong\Tor^U_n(E,\>U\ot_RM)$ and
$\Ext_R^n(M,D)\cong\Ext_U^n(U\ot_RM,\>D)$.
 In the situation at hand, $D=\Tor_1^R(U,A)$ is a left $U$\+module,
so any $R$\+module morphism $M\rarrow D$ vanishes, since $U\ot_RM=0$.
 Finally, we have $\Ext^1_R(M,D)=\Ext^1_U(U\ot_RM,\>D)=0$, since
$\Tor_1^R(U,M)=0$ and $U\ot_RM=0$.
\end{proof}

\begin{prop} \label{u-contramodule-category}
 Assume that\/ $\pd{}_RU\le 1$.  Then \par
\textup{(a)} 
 the full subcategory $R\modl_{u\ctra}$ is closed under the kernels,
cokernels, extensions, and products in $R\modl$.
 So $R\modl_{u\ctra}$ is an abelian category and the embedding functor
$R\modl_{u\ctra}\rarrow R\modl$ is exact; \par
\textup{(b)} assuming also that\/ $\Tor_1^R(U,U)=0$, the functor\/
$\Delta_u\:R\modl\rarrow R\modl_{u\ctra}$ is left adjoint to the fully
faithful embedding functor $R\modl_{u\ctra}\rarrow R\modl$.
\end{prop}

\begin{proof}
 Part~(a) is a particular case of~\cite[Proposition~1.1]{GL}
or~\cite[Theorem~1.2(a)]{Pcta}.
 The proof of part~(b) is based on Lemma~\ref{matlis-ext-1-lemma}(c)
and dual-analogous to the proof of
Proposition~\ref{u-comodule-category}(b);
cf.~\cite[Theorem~3.4]{PMat}.
\end{proof}

\begin{lem} \label{stays-divisible-torsionfree}
 Assume that\/ $\fd U_R\le1$, \ $\pd{}_RU\le1$, and\/ $Tor_1^R(U,U)=0$.
 Then\/ \par
\textup{(a)} for any $u$\+divisible left $R$\+module $B$, the left
$R$\+module\/ $\Gamma_u(B)$ is also $u$\+divisible; \par
\textup{(b)} for any $u$\+torsionfree left $R$\+module $A$, the left
$R$\+module\/ $\Delta_u(A)$ is also $u$\+torsionfree.
\end{lem}

\begin{proof}
 Let us prove part~(a).
 Following Lemma~\ref{torsionfree-divisible}(b), we have to check
that $\Ext^1_R(K^\bu,\Tor_1^R(K^\bu,B))=0$.
 Since $B$ is $u$\+divisible, we have $\Tor^R_0(K^\bu,B)=
(U/R)\ot_RB=0$, so the five-term exact
sequence~\eqref{main-tor-sequence} reduces to a four-term sequence.
 Furthermore, $\Ext_R^*(K^\bu,D)=0$ for any left $U$\+module~$D$.
 Thus it follows from~\eqref{main-tor-sequence} that
$\Ext^1_R(K^\bu,\Tor_1^R(K^\bu,B))=\Ext^1_R(K^\bu,B)=0$
(cf.\ the proof of Theorem~\ref{second-matlis}).
 The proof of part~(b) is dual-analogous.
\end{proof}

 In the category-theoretic terminology, a right adjoint functor to
the inclusion of a subcategory is called a \emph{coreflector}, and
a subcategory admitting such a functor is said to be
\emph{coreflective}.
 The following result is essentially well-known.

\begin{lem} \label{coreflective-Grothendieck}
 Let\/ $\sK$ be a category with colimits and\/ $\sA\subset\sK$ be
a coreflective full subcategory with the coreflector\/ $\Gamma\:\sK
\rarrow\sA$.
 Assume that there exists a regular cardinal~$\lambda$ such that
the coreflector\/ $\Gamma$ (say, viewed as a functor\/ $\sK\rarrow\sK$)
preserves $\lambda$\+filtered direct limits.
 Then \par
\textup{(a)} if the category\/ $\sK$ is locally presentable, then
the category\/ $\sA$ is locally presentable as well; \par
\textup{(b)} if\/ $\sK$ is a Grothendieck abelian category and
the full subcategory\/ $\sA$ is closed under kernels in\/ $\sK$,
then\/ $\sA$ is a Grothendieck abelian category, too.
 In this case, if $J$ is an injective cogenerator of\/ $\sK$, then\/
$\Gamma(J)$ is an injective cogenerator of\/~$\sA$.
\end{lem}

\begin{proof}
 Part~(a) can be obtained as a particular case
of~\cite[Exercise~2.m]{AR} (which is provable
using~\cite[Theorem~2.72 and Lemma~2.76]{AR}).
 Indeed, the full subcategory $\sA\subset\sK$ is the inverter of
the morphism of functors (adjunction counit) $\Gamma\rarrow\Id_\sK$.

 Alternatively, one observes that $\sA$ is closed under colimits in
$\sK$ (as any coreflective full subcategory).
 Hence the coreflector $\Gamma$ preserves $\lambda$\+filtered direct
limits as a functor $\sK\rarrow\sK$ if and only if it does so as
a functor $\sK\rarrow\sA$.
 Furthermore, it follows that all the objects of $\sA$ are presentable
(``have presentability ranks''), and in view of~\cite[Theorem~1.20]{AR}
it remains to show that the category $\sA$ has a strongly generating
set of objects.
 In part~(b), the full subcategory $\sA\subset\sK$ is closed under
kernels and all colimits; hence $\sA$ is abelian with exact functors
of direct limit.
 Once again, in order to show that the category $\sA$ is Grothendieck,
it remains to check that it has a set of generators (and it suffices
to do so in the context of part~(a)).

 Let $\kappa\ge\lambda$ be a regular cardinal such that the category
$\sK$ is locally $\kappa$\+presentable.
 Denote by $\sK^{<\kappa}$ a set of representatives of the isomorphism
classes of $\kappa$\+presentable objects in~$\sK$, and let $\sG$
denote the set of objects $\Gamma(G)\in\sA$, where $G\in\sK^{<\kappa}$.
 We claim that $\sG$ is a (strongly) generating set of objects in~$\sA$.

 Indeed, let $M\in\sA$ be an object; then we have $M\cong\Gamma(M)$.
 Let $(G_\alpha)$ be a $\kappa$\+filtered diagram of objects in
$\sK^{<\kappa}$ such that $M=\varinjlim_\alpha^\sK G_\alpha$ (where
the upper index denotes the category in which the colimit is taken).
 Then we have $M\cong\Gamma(M)=\varinjlim_\alpha^\sA\Gamma(G_\alpha)$.
 So $M$ is the direct limit of a diagram of objects from $\sG$ in~$\sA$.
 In particular, it follows that $M$ is a quotient of a coproduct of
copies of objects from~$\sG$.

 Finally, in the context of part~(b), the functor $\Gamma$ is right
adjoint to an exact functor, so takes injective objects of $\sK$ to
injective objects of~$\sA$.
 To show that $\Gamma(J)$ is an injective cogenerator of $\sA$ when
$J$ is an injective cogenerator of $\sK$, it suffices to observe that
$\Hom_\sA(M,\Gamma(J))=\Hom_\sK(M,J)\ne0$ when $0\ne M\in\sA$.
\end{proof}

\begin{rem}
 Assuming Vop\v enka's principle, one can drop the assumption of
existence of a cardinal~$\lambda$ in
Lemma~\ref{coreflective-Grothendieck}.
 This is the result of~\cite[Corollary~6.29]{AR}.
\end{rem}

 Now we return to the algebraic setting of this section.

\begin{cor} \label{u-comodule-category-Grothendieck}
 Assume that\/ $\Tor_1^R(U,U)=0$ and\/ $\fd U_R\le1$.
 Then $R\modl_{u\co}$ is a Grothendieck abelian category.
 If $J$ is an injective cogenerator of the abelian category
$R\modl$, then\/ $\Gamma_u(J)$ is an injective cogenerator of
$R\modl_{u\co}$.
\end{cor}

\begin{proof}
 This is a particular case of Lemma~\ref{coreflective-Grothendieck}(b).
 Indeed, by Proposition~\ref{u-comodule-category}(a), the full
subcategory $\sA=R\modl_{u\co}$ is closed under kernels in $\sK=R\modl$,
and by Proposition~\ref{u-comodule-category}(b), the full subcategory
$R\modl_{u\co}$ is coreflective in $R\modl$ with the coreflector
$\Gamma_u$ computable as $\Gamma_u=\Tor_1^R(K^\bu,{-})$.
 Viewed as a functor $R\modl\rarrow R\modl$, this Tor functor clearly
preserves direct limits.

 This suffices to prove the corollary.
 But let us mention that the module category $\sK=R\modl$ is locally
finitely presentable.
 So a set of generators of the category $\sA=R\modl_{u\co}$
can be constructed by applying the functor $\Gamma_u$ to
a representative set of isomorphism classes of finitely presentable
left $R$\+modules.
\end{proof}

\begin{lem} \label{u-contramodule-category-loc-pres}
 Assume that\/ $\Tor_1^R(U,U)=0$ and\/ $\pd{}_RU\le1$.
 Then $R\modl_{u\ctra}$ is a locally presentable abelian category with
a projective generator\/ $\Delta_u(R)\in R\modl_{u\ctra}$.
\end{lem}

\begin{proof}
 Following~\cite[Example~4.1(1\+2)]{PR}
or~\cite[Example~1.3(4)]{Pper}, if $\lambda$~is a regular
cardinal such that the left $R$\+module $U$ is $\lambda$\+presentable
(i.~e., isomorphic to the cokernel of a morphism of free left
$R$\+modules with less than $\lambda$~generators), then the category
$R\modl_{u\ctra}$ is locally $\lambda$\+presentable.
 Since the functor $\Delta_u$ is left adjoint to an exact (fully
faithful) functor $R\modl_{u\ctra}\rarrow R\modl$, it takes projective
left $R$\+modules to projective $u$\+contramodule left $R$\+modules.
 Finally, one has $\Hom_R(\Delta_u(R),C)=\Hom_R(R,C)=C\ne0$ for any
object $0\ne C\in R\modl_{u\ctra}$.
\end{proof}

\Section{The Endomorphism Ring of the Two-Term Complex $(R\to U)$}

 According to the discussion in~\cite[Section~1.1 in
the introduction]{PR}, \cite[Section~6.3]{PS}, and~\cite[Examples~1.2(4)
and~1.3(4)]{Pper}, under the assumptions of
Lemma~\ref{u-contramodule-category-loc-pres}
the abelian category $\sB=R\modl_{u\ctra}$ with its
natural projective generator $P=\Delta_u(R)$ can be described as
the category of modules over an additive monad $\boT_u$ on
the category of sets.
 For any set $X$, the coproduct $P^{(X)}$ of $X$ copies of the object
$P$ in the category $\sB$ can be computed as $P^{(X)}=\Delta_u(R^{(X)})$,
where $R^{(X)}=R[X]$ is the free $R$\+modules with generators indexed
by~$X$.
 The monad $\boT_u$ assigns to every set $X$ the set
$\Hom_\sB(P,P^{(X)})=\Delta_u(R^{(X)})$.
 In particular, to a one-element set~$*$, the monad $\boT_u$ assigns
the underlying set of the $R$\+module $P=\Delta_u(R)$.
 In fact $P=\boT_u(*)\in\boT_u\modl\cong\sB$ is the free
$\boT_u$\+module with one generator.

 For any additive monad $\boT$ on the category of sets, the set
$\boT(*)$ has a natural associative ring structure.
 This is the ring of endomorphisms of the forgetful functor
$\boT\modl\rarrow\Ab$.
 In particular, the ring $\R=\boT_u(*)$ can be computed as
the opposite ring to the ring of endomorphisms
$$
 \Delta_u(R)=\Ext^1_R(K^\bu,R)=\Hom_{\sD^\bb(R\modl)}(K^\bu,R[1])
 \cong\Hom_{\sD^\bb(R\modl)}(K^\bu,K^\bu).
$$
of the object $K^\bu$ in the derived category of left $R$\+modules.
 Notice that the right action of the ring $R$ by endomorphisms of
the object $K^\bu\in\sD^\bb(R\modl)$ in the derived category induces
a natural ring homomorphism $R\rarrow\R$.

 In the next lemma we discuss the particular case of a commutative
ring~$R$.

\begin{lem} \label{endomorphisms-commutative}
 Let $u\:R\rarrow U$ be an epimorphism of commutative rings such
that $\Tor_1^R(U,U)=0$.
 Then the ring\/ $\R=\Hom_{\sD^\bb(R\modl)}(K^\bu,K^\bu)$ is commutative.
 In particular, if $u$~is injective, then the ring\/
$\R=\Hom_R(U/R,U/R)$ is commutative.
\end{lem}

\begin{proof}
 This is a generalization of~\cite[Proposition~3.1]{PMat}.
 Let us prove the equivalent assertion that the ring\/
$\R=\Hom_{\sD^\bb(R\modl)}(K^\bu[-1],K^\bu[-1])$ is commutative
(where $K^\bu[-1]$ is the complex $R\rarrow U$ with the term $R$
placed in the cohomological degree~$0$ and the term $U$ placed
in the cohomological degree~$1$).
 Denote by $\sK$ the full subcategory in $\sD^\bb(R\modl)$ consisting
of the single object $K^\bu[-1]$ (and all the objects isomorphic to it).
 Then the functor of truncated tensor product
$$
 L^\bu\mathbin{\bar\otimes}M^\bu=
 \tau_{\ge-1}(L^\bu\ot_R^\boL M^\bu)
$$
defines a unital tensor (monoidal) category structure on
the category $\sK$ with the unit object $K^\bu[-1]$.
 In other words, there is a natural isomorphism $K^\bu[-1]\mathbin{\bar
\otimes}K^\bu[-1]\cong K^\bu[-1]$ transforming both the endomorphisms
$f\mathbin{\bar\otimes}\id$ and ${\id}\mathbin{\bar\otimes}f$ into
the endomorphism~$f$ for any $f\:K^\bu[-1]\rarrow K^\bu[-1]$.
 The commutativity of endomorphisms follows formally from that
(see the computation in~\cite{PMat}).

 When $u$~is a homological epimorphism, one does not need to truncate
the tensor product, so one can use the functor $\otimes_R^\boL$
instead of~$\bar\otimes$.
 When $u$~is an injective epimorphism, it suffices to consider
the full subcategory spanned by the object $K=U/R$ in $R\modl$ and
the functor $\Tor^R_1({-},{-})$ in the role of the tensor product
operation.
 Then one has to use the natural isomorphism $\Tor^R_1(K,K)\cong K$.
\end{proof}

 The next lemma shows that the second assertion of 
Lemma~\ref{endomorphisms-commutative} also holds for noninjective ring
epimorphisms~$u$ of projective dimension~$\le1$.

\begin{lem}
 Let $u\:R\rarrow U$ be an epimorphism of associative rings such that\/
$\Tor_1^R(U,U)=0$ and\/ $\pd{}_RU\le1$.
 Then the associative ring homomorphism\/
$$
 \Hom_{\sD^\bb(R\modl)}(K^\bu,K^\bu)\lrarrow\Hom_R(U/R,U/R)
$$
produced by applying the degree-zero cohomology functor $H^0\:
\sD^\bb(R\modl)\rarrow R\modl$ to the complex $K^\bu\in\sD^\bb(R\modl)$
is surjective.
 In particular, if the ring $R$ is commutative, then so is the ring\/
$\Hom_R(U/R,U/R)$.
\end{lem}

\begin{proof}
 Let $I\subset R$ be the kernel of the map~$u$.
 Then we have a natural distinguished triangle
$$
 I[1]\lrarrow K^\bu\lrarrow U/R\lrarrow I[2]
$$
in $\sD^\bb(R\bimod R)$, and we can also consider it as a distinguished
triangle in $\sD^\bb(R\modl)$.
 Applying the functor $\Hom_{\sD^\bb(R\modl)}(K^\bu,{-}[*])$ to this
triangle, we see that the map $\Hom_{\sD^\bb(R\modl)}(K^\bu,K^\bu)
\rarrow\Hom_{\sD(R\modl)}(K^\bu,U/R)$ is surjective, because
$\Hom_{\sD(R\modl)}(K^\bu,I[2])=\Ext^2_R(K^\bu,I)\cong\Ext^2_R(U,I)=0$
by Lemma~\ref{tor-ext-with-K}(b) and since $\pd{}_RU\le1$.
 Finally, we have $\Hom_{\sD(R\modl)}(K^\bu,U/R)=\Ext^0_R(K^\bu,U/R)
\cong\Hom_R(U/R,U/R)$ by Lemma~\ref{tor-ext-with-K}(c).

 This proves the first assertion of the lemma.
 The second one follows from the first one together with the first
assertion of Lemma~\ref{endomorphisms-commutative}.
\end{proof}

\Section{When is the Class of Torsion Modules Hereditary?}
\label{when-hereditary-secn}

 Notice that every left $u$\+comodule is $u$\+torsion, but the converse
implication does not need to be true.
 The torsion class of all $u$\+torsion left $R$\+modules does \emph{not}
need to be hereditary, i.~e., a submodule of a $u$\+torsion $R$\+module
does not need to be $u$\+torsion.
 In fact, if $\Tor_1^R(U,U)=0$ and $\fd U_R\le 1$, then any one of
the mentioned two properties holds if and only if $U$ is a flat 
right $R$\+module.

\begin{lem} \label{implies-flatness}
 Assume that $\Tor_1^R(U,U)=0$ and $\fd U_R\le1$.
 Then the following conditions are equivalent:
\begin{enumerate}
\item all $u$\+torsion left $R$\+modules are $u$\+comodules;
\item all quotient $R$\+modules of left $u$\+comodules
are $u$\+comodules;
\item all $R$\+submodules of left $u$\+comodules are $u$\+comodules;
\item all $R$\+submodules of $u$\+torsion left $R$\+modules
are $u$\+torsion;
\item all $R$\+submodules of left $u$\+comodules are $u$\+torsion;
\item the right $R$\+module $U$ is flat.
\end{enumerate}
\end{lem}

\begin{proof}
 (1)\,$\Longrightarrow$\,(2)
 By the definition, all $u$\+comodules are $u$\+torsion.
 Hence (1)~means that the classes of left $u$\+comodules and $u$\+torsion
left $R$\+modules coincide.
 Since the class of $u$\+torsion $R$\+modules is clearly closed under
quotients, (2)~follows.

 (2)\,$\Longleftrightarrow$\,(3) holds because the class of all left
$u$\+comodules is closed under kernels and cokernels of morphisms (by
Proposition~\ref{u-comodule-category}(a)).

 (3)\,$\Longrightarrow$\,(5) and (4)\,$\Longrightarrow$\,(5) are obvious.
  
 (5)\,$\Longrightarrow$\,(6)
 Let $A$ be a left $R$\+module.
 From the exact sequence~\eqref{main-tor-sequence} we see that
 the left $R$\+module $\Tor_1^R(U,A)$ is a submodule of
the left $R$\+module $\Tor_1^R(K^\bu,A)$.
 By Lemma~\ref{matlis-tor-1-lemma}(c), $\Tor_1^R(K^\bu,A)=\Gamma_u(A)$ is
a left $u$\+comodule.
 Under~(5), it follows that the left $R$\+module $\Tor_1^R(U,A)$ is
$u$\+torsion.
 Being simultaneously a left $U$\+module, it follows that
$\Tor_1^R(U,A)=0$.

 (6)\,$\Longrightarrow$\,(1) and (6)\,$\Longrightarrow$\,(4) are obvious.
\end{proof}

 Examples of noncommutative homological ring epimorphisms of projective
dimension~$1$ (on both sides) that are not flat (on either side)
do exist.
 Let $k$~be a field, $k[x]$~be the polynomial ring in one variable~$x$
with the coefficients in~$k$, and $kx\subset k[x]$ be
the one-dimensional $k$\+vector subspace spanned by~$x$.
 Then the embedding of matrix rings
$R=\left(\begin{smallmatrix} k & k\oplus kx \\ 0 & k \end{smallmatrix}
\right)\rarrow \left(\begin{smallmatrix} k[x] & k[x] \\ k[x] & k[x]
\end{smallmatrix}\right)=U$ is an injective ring epimorphism such that
$\Tor_1^R(U,U)=0$ and $\pd{}_RU=\pd U_R=\fd{}_RU=\fd U_R=1$
(cf.\ Section~\ref{kronecker-quiver-secn}).

 On the other hand, the following theorem holds true for epimorphisms of
commutative rings.

\begin{thm}
 If $u\:R\rarrow U$ is an epimorphism of commutative rings such that
$\Tor_1^R(U,U)=0$ and $\pd{}_RU\le1$, then $U$ is a flat
$R$\+module.
\end{thm}

\begin{proof}
 The argument is based on some results from the papers~\cite{Hrb,AH}.
 Assume first that $u$~is injective.
 Then $U\oplus U/R$ is a $1$\+tilting
$R$\+module~\cite[Theorem~3.5]{AS}, hence
$C=\Hom_\boZ(U\oplus U/R,\>\boQ/\boZ)$ is a $1$\+cotilting $R$\+module
of cofinite type~\cite[Theorems~15.2 and~15.18]{GT}.
 The $1$\+cotilting class associated with $C$ consists of all
the $R$\+submodules of $U$\+modules; in other words, it is what we call
the class of all $u$\+torsionfree $R$\+modules.
 Hence the torsion class in the $1$\+cotilting torsion pair associated
with $C$ is the class of all $u$\+torsion $R$\+modules.
 According to~\cite[Proposition~3.11]{Hrb}, any $1$\+cotilting torsion
pair of cofinite type in the category of modules over a commutative ring
is hereditary.
 By Lemma~\ref{implies-flatness}\,(4)\,$\Longrightarrow$\,(6), it follows
that $\fd{}_RU=0$.

 In the general case of a (not necessarily injective) homological
epimorphism of commutative rings $u\:R\rarrow U$ with $\pd{}_RU\le1$,
one has to use silting theory instead of tilting theory.
 The $R$\+module $U\oplus U/R$ is $1$\+silting by~\cite[Example~6.5]{MS},
and a $2$\+term projective resolution of the complex $U\oplus K^\bu$ is
the related silting complex.
 Hence $C=\Hom_\boZ(U\oplus U/R,\>\boQ/\boZ)$ is a cosilting $R$\+module
of cofinite type~\cite[Corollary~3.6]{AH}.
 The cosilting class associated with $C$ consists of all
the $u$\+torsionfree $R$\+modules, and the torsion class in
the cosilting torsion pair is the class of all $u$\+torsion
$R$\+modules.
 By~\cite[Lemma~4.2]{AH}, any cosilting torsion pair of cofinite type
in the category of modules over a commutative ring is hereditary.
 Once again, by Lemma~\ref{implies-flatness}\,(4)\,$\Longrightarrow$\,(6)
we can conclude that $U$ is a flat $R$\+module.
\end{proof}

\Section{Triangulated Matlis Equivalence}
\label{triangulated-matlis-secn}

 Let $u\:R\rarrow U$ be a homological epimorphism of associative rings,
that is a ring homomorphism such that the natural map of
$U$\+$U$\+bimodules $U\ot_R U\rarrow U$ is an isomorphism and
$\Tor^R_i(U,U)=0$ for all $i>0$.
  Then, according to~\cite[Theorem~4.4]{GL}, \cite[Theorem~3.7]{Pa}, 
\cite[Lemma in Section~4]{NS}, the restriction of scalars with respect
to~$u$ is a fully faithful functor between the unbounded derived
categories $\sD(U\modl)\rarrow\sD(R\modl)$.
 We denote this functor, acting between the bounded or unbounded
derived categories, by
$$
 u_*\:\sD^\star(U\modl)\lrarrow\sD^\star(R\modl),
$$
where $\star=\bb$, $+$, $-$, or~$\varnothing$ is a derived category
symbol.

  In the case of the unbounded derived categories ($\star=\varnothing$),
the functor~$u_*$ has a left adjoint functor
$\boL u^*\:\sD(R\modl)\rarrow\sD(U\modl)$ and a right adjoint functor
$\boR u^!\:\sD(R\modl)\rarrow\sD(U\modl)$.
 When $U$ is a right $R$\+module of finite flat dimension, the functor
$\boL u^*$ also acts between bounded derived categories,
$$
 \boL u^*\:\sD^\star(R\modl)\lrarrow\sD^\star(U\modl).
$$
 When $U$ is a left $R$\+module of finite projective dimension,
the functor $\boR u^!$ acts between bounded derived categories,
$$
 \boR u^!\:\sD^\star(R\modl)\lrarrow\sD^\star(U\modl).
$$
 Since the triangulated functor $u_*$ is fully faithful, its left and
right adjoints $\boL u^*$ and $\boR u^!$ are Verdier quotient
functors~\cite[Proposition~I.1.3]{GZ}.

\begin{thm} \label{two-triangulated-equivalences}
\textup{(a)} Assume that\/ $\fd U_R\le1$.
 Then the kernel of the functor\/ $\boL u^*\:\sD^\star(R\modl)\rarrow
\sD^\star(U\modl)$ coincides with the full subcategory\/
$\sD^\star_{u\co}(R\modl)\subset\sD^\star(R\modl)$ of all complexes
of left $R$\+modules with $u$\+comodule cohomology modules.
 Hence for every symbol\/ $\star=\bb$, $+$, $-$, or\/~$\varnothing$,
we have a triangulated equivalence
$$
 \sD^\star(R\modl)/u_*\sD^\star(U\modl)
 \,\cong\,\sD^\star_{u\co}(R\modl).
$$
\textup{(b)} Assume that\/ $\pd{}_RU\le1$.
 Then the kernel of the functor\/ $\boR u^!\:\sD^\star(R\modl)\rarrow
\sD^\star(U\modl)$ coincides with the full subcategory\/
$\sD^\star_{u\ctra}(R\modl)\subset\sD^\star(R\modl)$ of all complexes
of left $R$\+modules with $u$\+contramodule cohomology modules.
 Hence for every symbol\/ $\star=\bb$, $+$, $-$, or\/~$\varnothing$,
we have a triangulated equivalence
$$
 \sD^\star(R\modl)/u_*\sD^\star(U\modl)
 \,\cong\,\sD^\star_{u\ctra}(R\modl).
$$
\end{thm}

\begin{proof}
 Part~(a): the functor $\boL u^*$ is constructed as the derived tensor
product $\boL u^*(A^\bu) = U\ot_R^\boL A^\bu$ for any complex
of left  $R$\+modules~$A^\bu$.
 In particular, when $\fd U_R\le1$, we have short exact sequences of
cohomology
$$
 0\lrarrow U\ot_R H^n(A^\bu)\lrarrow H^n(\boL u^*(A^\bu))\lrarrow
 \Tor_1^R(U,H^{n+1}(A^\bu))\lrarrow0
$$
for any complex $A^\bu\in\sD^\star(R\modl)$ and all $n\in\boZ$.
 It follows immediately that $\boL u^*(A^\bu)=0$ if and only if
$H^n(A^\bu)\in R\modl_{u\co}$ for all $n\in\boZ$.

 Part~(b): the functor $\boR u^!$ is constructed as the derived
homomorphisms $\boR u^!(B^\bu)=\boR\Hom_R(U,B^\bu)$ for any
complex of left $R$\+modules~$B^\bu$.
 In particular, when $\pd{}_RU\le1$, we have short exact sequences
of cohomology
$$
 0\rarrow\Ext^1_R(U,H^{n-1}(B^\bu))\lrarrow H^n(\boR u^!(B^\bu))
 \lrarrow\Hom_R(U,H^n(B^\bu))\lrarrow0
$$
for any complex $B^\bu\in\sD^\star(R\modl)$ and all $n\in\boZ$.
 It follows immediately that $\boR u^!(B^\bu)=0$ if and only if
$H^n(B^\bu)\in R\modl_{u\ctra}$ for all $n\in\boZ$.
\end{proof}

\begin{cor} \label{one-triangulated-equivalence}
 Assume that\/ $\fd U_R\le1$ and\/ $\pd{}_RU\le1$.
 Then for every symbol\/ $\star=\bb$, $+$, $-$, or\/~$\varnothing$
there is a triangulated equivalence
$$
 \sD^\star_{u\co}(R\modl)\,\cong\,\sD^\star_{u\ctra}(R\modl)
$$
provided by the mutually inverse functors\/ $\boR\Hom_R(K^\bu[-1],{-})\:
\sD^\star_{u\co}(R\modl)\rarrow\sD^\star_{u\ctra}(R\modl)$
and $K^\bu[-1]\ot_R^\boL{-}\:\sD^\star_{u\ctra}(R\modl)\rarrow
\sD^\star_{u\co}(R\modl)$.
\end{cor}

\begin{proof}
 More generally, in the context of
Theorem~\ref{two-triangulated-equivalences}(a), the functor
$\sD^\star(R\modl)\allowbreak\rarrow\sD^\star_{u\co}(R\modl)$
right adjoint to the embedding
$\sD^\star_{u\co}(R\modl)\rarrow\sD^\star(R\modl)$
is computed as $K^\bu[-1]\ot_R^\boL{-}$.
 Similarly, in the context of
Theorem~\ref{two-triangulated-equivalences}(b), the functor
$\sD^\star(R\modl)\rarrow\sD^\star_{u\ctra}(R\modl)$ left adjoint
to the embedding $\sD^\star_{u\ctra}(R\modl)\rarrow\sD^\star(R\modl)$
is computed as $\boR\Hom_R(K^\bu[-1],{-})$
(cf.~\cite[Proposition~4.4]{PMat}).
\end{proof}

 The results of this section can be expressed by existence of
the following \emph{recollement} of triangulated categories for
any homological ring epimorphism $u\:R\rarrow U$ such that 
$\fd U_R\le1$ and $\pd{}_RU\le1$:
\begin{equation} \label{general-matlis-recollement}
\begin{tikzcd}
 \sD^\star(U\modl) \arrow[r, tail] &
 \sD^\star(R\modl)
 \arrow[l, two heads, bend left=30]
 \arrow[l, two heads, bend right=30]
 \arrow[r, two heads] &
 {\sD^\star_{u\co}(R\modl)=\sD^\star_{u\ctra}(R\modl)\qquad\qquad}
 \arrow[l, tail, bend left=20]
 \arrow[l, tail, bend right=20]
\end{tikzcd}
\mkern-72mu
\end{equation}
 Here the arrows with a tail denote fully faithful triangulated functors,
while the arrows with two heads denote triangulated Verdier quotient
functors.
 The two leftmost curvilinear arrows are adjoint on the left and on
 the right to the leftmost straight arrow, while the two rightmost
curvilinear arrows are adjoint on the left and on the right to
the rightmost straight arrow.
 The image of the leftmost straight arrow is the kernel of the rightmost
straight arrow, and similarly with the two pairs of curvilinear arrows.

 The three leftmost arrows are the functors $\boL u^*$, \,$u_*$, and
$\boR u^!$.
 The two rightmost curvilinear arrows are the inclusions of
the full subcategories $\sD^\star_{u\co}(R\modl)$ and
$\sD^\star_{u\ctra}(R\modl)$ into $\sD^\star(R\modl)$.

\Section{Two Fully Faithful Triangulated Functors}

 In addition to the assumptions on the projective and flat dimension of
the left and right $R$\+module $U$ that we used above, the results
of this section require certain assumptions about the properties of
injective and projective left $R$\+modules vis-\`a-vis the homological
ring homomorphism $u\:R\rarrow U$.
 Specifically, these are the assumptions that injective left $R$\+modules
are $u$\+special and projective left $R$\+modules are $u$\+cospecial,
or in other words, the left $R$\+modules $\Tor_0^R(K^\bu,J)=U/R\ot_RJ$
and $\Ext^0_R(K^\bu,F)=\Hom_R(U/R,F)$ vanish for all injective left
$R$\+modules $J$ and projective left $R$\+modules~$F$
(cf.\ Lemmas~\ref{tor-ext-with-K}(c) and~\ref{special-and-cospecial}).

\begin{thm} \label{fully-faithful-triangulated}
\textup{(a)} Assume that\/ $\fd U_R\le1$ and $(U/R)\ot_R J=0$ for all
injective left $R$\+modules~$J$.
  Then, for any conventional derived category symbol\/ $\star=\bb$, $+$,
$-$, or\/~$\varnothing$, the triangulated functor
$$
 \sD^\star(R\modl_{u\co})\lrarrow\sD^\star(R\modl)
$$
induced by the exact embedding of abelian categories $R\modl_{u\co}
\rarrow R\modl$ is fully faithful, and its essential image coincides with
the full subcategory
$$
 \sD^\star_{u\co}(R\modl)\,\subset\,\sD^\star(R\modl),
$$
providing an equivalence of triangulated categories
$$
 \sD^\star(R\modl_{u\co})\,\cong\,\sD^\star_{u\co}(R\modl).
$$ \par
\textup{(b)} Assume that\/ $\pd{}_RU\le1$ and\/ $\Hom_R(U/R,F)=0$ for
all projective left $R$\+modules~$F$.
 Then, for any conventional derived category symbol\/ $\star=\bb$, $+$,
$-$, or\/~$\varnothing$, the triangulated functor
$$
 \sD^\star(R\modl_{u\ctra})\lrarrow\sD^\star(R\modl)
$$
induced by the exact embedding of abelian categories $R\modl_{u\ctra}
\rarrow R\modl$ is fully faithful, and its essential image coincides with
the full subcategory
$$
 \sD^\star_{u\ctra}(R\modl)\,\subset\,\sD^\star(R\modl),
$$
providing an equivalence of triangulated categories
$$
 \sD^\star(R\modl_{u\ctra})\,\cong\,\sD^\star_{u\ctra}(R\modl).
$$
\end{thm}

\begin{proof}
 This is an application of the general technique formulated
in~\cite[Theorem~6.4 and Proposition~6.5]{PMat}.
 Let us explain part~(b).
 The pair of functors $\Ext^i_R(K^\bu,{-})$, \ $i=0$,~$1$,
is a cohomological functor between the abelian categories $R\modl$
and $R\modl_{u\ctra}$, that is, for every short exact sequence of
left $R$\+modules $0\rarrow A\rarrow B\rarrow C\rarrow0$ there is
a short exact sequence of left $u$\+contramodules
(cf.\ Lemmas~\ref{tor-ext-with-K}(a\+b)
and~\ref{matlis-ext-1-lemma}(a,c))
\begin{multline*}
 0\lrarrow\Ext^0_R(K^\bu,A)\lrarrow\Ext^0_R(K^\bu,B)\lrarrow
\Ext^0_R(K^\bu,C) \\ \lrarrow\Ext^1_R(K^\bu,A)\lrarrow
\Ext^1_R(K^\bu,B)\lrarrow\Ext^1_R(K^\bu,C)\lrarrow0.
\end{multline*}
 Since, by our assumption, the functor $\Ext^0_R(K^\bu,{-})$
annihilates projective left $R$\+modules, it follows that
our cohomological functor $\Ext^*_R(K^\bu,{-})$ is the left derived
functor of the functor $\Delta=\Delta_u=\Ext^1_R(K^\bu,{-})\:R\modl
\rarrow R\modl_{u\ctra}$, that is $\boL_1\Delta_u=\Ext^0_R(K^\bu,{-})$
and $\boL_i\Delta_u=0$ for $i>1$.

 By Proposition~\ref{u-contramodule-category}(b), the functor $\Delta_u$
is left adjoint to the exact, fully faithful embedding functor
$R\modl_{u\ctra}\rarrow R\modl$, so we are in the setting
of~\cite[Theorem~6.4]{PMat}.
 It remains to point out that $\boL_1\Delta_u(B)=\Ext^0_R(K^\bu,B)=0$
for all left $u$\+contramodules~$B$.
 Notice that the class $R\modl_{\Delta\adj}=\Ker(\boL_{>0}\Delta)$ of
$\Delta$\+adjusted left $R$\+modules, playing a key role in
the argument in~\cite[Section~6]{PMat}, is nothing but the class of
$u$\+cospecial left $R$\+modules in our context, according to
Lemma~\ref{special-and-cospecial}.

 Similarly, in part~(a) one observes that the pair of functors
$\Tor_i^R(K^\bu,{-})$, \ $i=0$,~$1$ is a homological functor between
the abelian categories $R\modl$ and $R\modl_{u\co}$, hence, whenever
the functor $\Tor_1^R(K^\bu,{-})$ annihilates injective left
$R$\+modules, it is the right derived functor of the functor
$\Gamma=\Gamma_u=\Tor_1^R(K^\bu,{-})\:R\modl\rarrow R\modl_{u\co}$,
that is $\boR^1\Gamma_u=\Tor_0^R(K^\bu,{-})$ and
$\boR^i\Gamma=0$ for $i>1$.
 It remains to point out that $\boR^1\Gamma_u(A)=\Tor_0^R(K^\bu,A)=0$
for all left $u$\+comodules~$A$.
 As above, we notice that the class $R\modl_{\Gamma\adj}=
\Ker(\boR^{>0}\Gamma)$ of $\Gamma$\+adjusted left $R$\+modules
is just the class of $u$\+special left $R$\+modules discussed in
Section~\ref{second-additive-equivalence}.
\end{proof}

\begin{rem}
 Conversely, if $\fd U_R\le1$ and the triangulated functor
$\sD^\bb(R\modl_{u\co})\allowbreak\rarrow\sD^\bb(R\modl)$ is
fully faithful, then $(U/R)\ot_RJ=0$ for all injective
left $R$\+modules~$J$.
 A proof of this can be found in~\cite[Lemma~3.9 and
Proposition~4.2]{CX} (cf.~\cite[Remark~6.8]{PMat}).
 Similarly, if $\pd{}_RU\le1$ and the triangulated functor
$\sD^\bb(R\modl_{u\ctra})\rarrow\sD^\bb(R\modl)$ is fully faithful, then
$\Hom_R(U/R,F)=0$ for all projective left $R$\+modules~$F$
\cite[Lemma~3.9 and Proposition~4.1]{CX}.
\end{rem}

 The following result can be also found in~\cite[Corollary~4.4]{CX}.

\begin{cor} \label{matlis-derived-equivalence-cor}
 Let $u\:R\rarrow U$ be a homological ring epimorphism.
 Assume that\/ $\fd U_R\le1$ and\/ $\pd{}_RU\le1$.
 Suppose further that $(U/R)\ot_RJ=0$ for all injective left $R$\+modules
$J$ and\/ $\Hom_R(U/R,F)=0$ for all projective left $R$\+modules~$F$.
 Then for every conventional derived category symbol\/ $\star=\bb$, $+$,
$-$, or\/~$\varnothing$, there is a triangulated equivalence between
the derived categories of the abelian categories $R\modl_{u\co}$ and
$R\modl_{u\ctra}$ of left $u$\+comodules and left $u$\+contramodules,
$$
 \sD^\star(R\modl_{u\co})\cong\sD^\star(R\modl_{u\ctra}).
$$
\end{cor}

\begin{proof}
 According to Corollary~\ref{one-triangulated-equivalence} and
Theorem~\ref{fully-faithful-triangulated}(a\+b), we have a chain of
triangulated equivalences
$$
 \sD^\star(R\modl_{u\co})\cong\sD^\star_{u\co}(R\modl)\cong
 \sD^\star_{u\ctra}(R\modl)\cong\sD^\star(R\modl_{u\ctra}).
$$
\end{proof}

 Under the assumptions of
Corollary~\ref{matlis-derived-equivalence-cor},
the recollement~\eqref{general-matlis-recollement} takes the form
\begin{equation} \label{particular-matlis-recollement}
\begin{tikzcd}
 \sD^\star(U\modl) \arrow[r, tail] &
 \sD^\star(R\modl)
 \arrow[l, two heads, bend left=30]
 \arrow[l, two heads, bend right=30]
 \arrow[r, two heads] &
 {\sD^\star(R\modl_{u\co})=\sD^\star(R\modl_{u\ctra})\qquad\qquad}
 \arrow[l, tail, bend left=20]
 \arrow[l, tail, bend right=20]
\end{tikzcd}
\mkern-72mu
\end{equation}
 In the recollement~\eqref{particular-matlis-recollement}, all the three
triangulated categories are derived categories of certain abelian
categories (and the third one is even the derived category of two
different abelian categories).

\begin{ex} \label{injective-epi-matlis-example}
 For any injective ring epimorphism $u\:R\rarrow U$, the conditions
$(U/R)\ot_RJ=0$ and $\Hom_R(U/R,F)=0$ hold for all injective left
$R$\+modules $J$ and all projective left $R$\+modules~$F$.
 Indeed, if $u$~is injective and $J$ is an injective left $R$\+module,
then any left $R$\+module morphism $R\rarrow J$ can be extended to
a left $R$\+module morphism $U\rarrow J$.
 Hence the left $R$\+module $J$ is $u$\+divisible (i.~e., a quotient
$R$\+module of a left $U$\+module).
 Thus $U/R\ot_RU=0$ implies $U/R\ot_RJ=0$.
 Similarly, the map $F\rarrow U\ot_RF$ is injective for any flat left
$R$\+module $F$, so $F$ is $u$\+torsionfree (i.~e., an $R$\+submodule
of a left $U$\+module).
 Therefore, $\Hom_R(U/R,U)=0$ implies $\Hom_R(U/R,F)=0$.
\end{ex}

\Section{Kronecker Quiver Example} \label{kronecker-quiver-secn}

 Let $k$ be an algebraically closed field, and let $R$ denote the path
algebra of the Kronecker quiver $\bullet\rightrightarrows\bullet$
over~$k$.
 So left $R$\+modules are pairs of $k$\+vector spaces $(V_1,V_2)$ endowed
with a pair of $k$\+linear maps $f_V$, $g_V\:V_1\rightrightarrows V_2$.
 The aim of this section is to describe the full subcategories of
comodules and contramodules for certain ring epimorphisms originating
from~$R$.

 We will interpret $R$ as the matrix ring $R=\left(\begin{smallmatrix}
k & k\oplus kx \\ 0 & k \end{smallmatrix}\right)$, where the element
$1\in k\oplus kx$ in the upper right corner acts in the quiver
representations by the map~$f_V$ and the element $x\in k\oplus kx$ acts
by the map~$g_V$.
 When the map~$f_V$ is invertible, the fraction $x=g_V/f_V$ is a linear
operator $V_1\rarrow V_1$ or $V_2\rarrow V_2$.
 The eigenvalues of this operator, if they happen to exist, can be
thought of as points of the projective line
$\boP^1(k)=k\cup\{\infty\}$ with the coordinate~$x$.

 Let $\boX\subset\boP^1(k)$ be a subset of points of the projective
line such that $\infty\in\boX$.
 Denote by $S_\boX=\boX^{-1}k[x]$ the localization of the ring of
polynomials $k[x]$ at the multiplicative subset generated by the elements
$x-\lambda$, \ $\lambda\in\boX\setminus\{\infty\}$.
 Consider the matrix ring $U_\boX=\left(\begin{smallmatrix}
S_\boX & S_\boX \\ S_\boX & S_\boX \end{smallmatrix}\right)$.
 Then there is a ring homomorphism $u_\boX\:R\rarrow U_\boX$ given
by the inclusion of the matrices.
 The map~$u_\boX$ is a homological ring epimorphism.
 The essential image of the functor of restriction of scalars
$u_\boX{}_*\:U_\boX\modl\rarrow R\modl$ consists of all the quiver
representations $(f_V,g_V)$ such that the map~$f_V$ is an isomorphism and
the map $g_V-\lambda f_V\:V_1\rarrow V_2$ is an isomorphism for all
$\lambda\in\boX\setminus\{\infty\}$.

 In particular, for $\boX=\{\infty\}$ we have $U_{\{\infty\}}=
\left(\begin{smallmatrix} k[x] & k[x] \\ k[x] & k[x]
\end{smallmatrix}\right)$.
 The essential image of the functor $u_{\{\infty\}}{}_*\:
U_{\{\infty\}}\modl\rarrow R\modl$ consists of all the quiver
representations $(f_V,g_V)$ such that the map~$f_V$ is invertible.
 For an arbitrary subset $\{\infty\}\in\boX\subset\boP^1(k)$, the morphism
$u_\boX\:R\rarrow U_\boX$ factorizes as the composition of two injective
homological ring epimorphisms $R\rarrow U_{\{\infty\}}\rarrow U_\boX$.
 The ring $U_\boX$ has both flat and projective dimension~$1$ both as
a left and as a right $R$\+module (as a left and right
$U_{\{\infty\}}$\+module, it has flat dimension~$0$ and
projective dimension~$1$).

 So the full subcategories of $u_\boX$\+comodules and
$u_\boX$\+contramodules in $R\modl$ are abelian.
 Specifically, let us say that a quiver representation $M=(f_M,g_M)$ is
a ``($\lambda=\infty$)\+comodule'' if $\Hom_R(M,V)=0=\Ext^1_R(M,V)$ for
any quiver representation $V=(f_V,g_V)$ with the map~$f_V$ invertible.
 A quiver representation $C=(f_C,g_C)$ is
a ``($\lambda=\infty$)\+contramodule'' if $\Hom_R(V,C)=0=\Ext^1_R(V,C)$ for
any such~$V$.

 More generally, we will say that a quiver representation $M$ is
an ``$\boX$\+comodule'' if $\Hom_R(M,V)=0=\Ext^1_R(M,V)$ for any
$V=(f_V,g_V)$ with the maps $f_V$ and $g_V-\lambda f_V$ invertible for all
$\lambda\in\boX$.
 A quiver representation $C$ is an ``$\boX$\+contramodule'' if
$\Hom_R(V,C)=0=\Ext^1_R(V,C)$ for any such~$V$.
 One can easily see that a quiver representation is an $\boX$\+comodule
if and only if the related left $R$\+module is a $u_\boX$\+comodule,
and similarly, a quiver representation is an $\boX$\+contramodule if
and only if the related left $R$\+module is a $u_\boX$\+contramodule.
 
 Assume first that $0\notin\boX$.
 Denote by $y$ the coordinate $1/x$ on $\boP^1(k)$ (so $y$~is a possible
eigenvalue of $f_V/g_V$), and let $\boY\subset\boA^1(k)$ denote the subset
of the affine line consisting of all points $\mu\in k$ such that
$\mu^{-1}\in\boX$.
 Consider the polynomial ring $k[y]$, and denote by $T_\boY=\boY^{-1}k[y]$
its localization at the multiplicative subset generated by
the elements $y-\mu$, \ $\mu\in\boY$.
 Denote by $v_\boY$ the ring epimorphism $k[y]\rarrow T_\boY$.
 In particular, if $\boX=\{\infty\}$, then $\boY=\{0\}$ and
$\boT_\boY=k[y,y^{-1}]$.

 The results of the following lemma are (easy) particular cases of
the theory developed in~\cite[Section~13]{Pcta}
and~\cite[Section~4 and/or~6]{BP}.

\begin{lem} \label{polynomials-co-contra}
\textup{(a)}
 If\/ $\boY=\{\mu\}$ is a one-point set, then a $k[y]$\+module $M$ is
a $v_{\{\mu\}}$\+comodule if and only if the operator $y-\mu$ is locally
nilpotent in $M$, i.~e., for every $m\in M$ there exists
$n\in\boZ_{\ge1}$ such that $(y-\mu)^nm=0$.
 For an arbitrary subset\/ $\boY\subset\boA^1(k)$, any $v_\boY$\+comodule
$M$ has a unique, functorial decomposition into a direct sum of
$v_{\{\mu\}}$\+comodules $M_\mu$ over $\mu\in\boY$, and any such direct
sum $M=\bigoplus_{\mu\in\boY}M_\mu$ of $v_{\{\mu\}}$\+comodules $M_\mu$
is a $v_\boY$\+comodule.
 The category $k[y]\modl_{v_\boY\co}$ of $v_\boY$\+comodules is thus
equivalent to the Cartesian product of the categories
$k[y]\modl_{v_{\{\mu\}}\co}$ over $\mu\in\boY$. \par
\textup{(b)}
 If\/ $\boY=\{\mu\}$ is a one-point set, then a $k[y]$\+module $C$ is
a $v_{\{\mu\}}$\+contramodule if and only if it admits $(y-\mu)$\+power
infinite summation operations in the sense of\/~\cite[Section~3]{Pcta}.
 For an arbitrary subset\/ $\boY\subset\boA^1(k)$, any
$v_\boY$\+contramodule $C$ has a unique, functorial decomposition into
a direct product of $v_{\{\mu\}}$\+contramodules $C^\mu$ over
$\mu\in\boY$, and any such direct product $C=\prod_{\mu\in\boY}C^\mu$
of $v_{\{\mu\}}$\+contramodules $C^\mu$ is a $v_\boY$\+contramodule.
 The category $k[y]\modl_{v_\boY\ctra}$ of $v_\boY$\+contramodules is
thus equivalent to the Cartesian product of the categories
$k[y]\modl_{v_{\{\mu\}}\ctra}$ over $\mu\in\boY$. \qed
\end{lem}

 The next proposition describes $\boX$\+comodule and
$\boX$\+contramodule quiver representations for $\boX\subset\boP^1(k)$,
\,$\infty\in\boX$, \,$0\notin\boX$.

\begin{prop} \label{outside-of-zero-prop}
 Let\/ $\boX$ be a subset in $\boP^1(k)=k\cup\{\infty\}$ containing
$\infty$ and not containing~$0$, and let $\boY\subset\boA^1_k=k$ be
the set of all~$\mu$ such that $\mu^{-1}\in\boX$.
  Then\par
\textup{(a)} a Kronecker quiver representation $M=(f_M,g_M)$ is
an\/ $\boX$\+comodule if and only if the map~$g_M$ is invertible
and the vector space $M_1\cong M_2$ with the linear operator $y=f_M/g_M$
is a $v_\boY$\+comodule; \par
\textup{(b)} a Kronecker quiver representation $C=(f_C,g_C)$ is
an\/ $\boX$\+contramodule if and only if the map~$g_C$ is invertible
and the vector space $C_1\cong C_2$ with the linear operator $y=f_C/g_C$
is a $v_\boY$\+contramodule.
\end{prop}

 It follows from Lemma~\ref{polynomials-co-contra} and
Proposition~\ref{outside-of-zero-prop} that the category of
$\boX$\+comodules decomposes as a Cartesian product of $\boX$ copies of
the category of vector spaces with a locally nilpotent operator~$z$,
and similarly, the category of $\boX$\+contramodules decomposes
as a Cartesian product of $\boX$ copies of the category of vector
spaces with $z$\+power infinite summation operations.

 Notice that it follows from Proposition~\ref{outside-of-zero-prop}
that the category of $\boX$\+comodules is \emph{equivalent} to
a torsion class (viewed as a full subcategory) in $k[y]\modl$.
 But a subrepresentation of an $\boX$\+comodule is \emph{not}
an $\boX$\+comodule, generally speaking (because the condition of
invertibility of the operator~$g_M$ is not preserved by the passage
to a subrepresentation), in agreement with the discussion in
Section~\ref{when-hereditary-secn}.

 The proof of Proposition~\ref{outside-of-zero-prop} given below
consists of several lemmas.

\begin{lem} \label{implies-invertibility}
 For any\/ $\boX$ not containing\/~$0$, one has: \par
\textup{(a)} in any\/ $\boX$\+comodule $M=(f_M,g_M)$, the map~$g_M$ is
invertible; \par
\textup{(b)} in any\/ $\boX$\+contramodule $C=(f_C,g_C)$, the map~$g_C$
is invertible.
\end{lem}

\begin{proof}
 We will prove part~(b).
 Assume that the operator~$g_C$ has a nonzero kernel $K_1\subset C_1$.
 Then there are two possibilities.
 If the kernel of the restriction of $f_C$ to $K_1$ is nonzero, then $C$
contains a copy of the injective representation $k\rightrightarrows0$
as a subrepresentation.
 In this case, for any nonzero representation $V$ with $f_V$~invertible
(hence $V_1\ne0$) there exists a nonzero morphism $V\rarrow
(k\rightrightarrows0)\rarrow C$.
 If the map $f_C|_{K_1}$ is injective, then $C$ contains
a nonzero subrepresentation $K=(K_1,f_C(K_1))$ with $f_K$ invertible and
$g_K=0$, hence $g_K-\lambda f_K$ is invertible for all $\lambda\in\boX$.
 In both cases, there exists a $U_\boX$\+module $V$ and a nonzero
morphism $V\rarrow C$, contradicting the assumption that $C$ is
an $\boX$\+contramodule.

 Assume that the map~$g_C$ is not surjective.
 Then the quiver representation $C$ has a quotient representation
$L=(C_1,C_2/g_C(C_1))$ with $g_L=0$ and $L_2=\coker(g_C)\ne0$.
 Once again, there are two possibilities.
 If the map~$f_L$ is not surjective, then $C$ has a projective
quotient representation $N=(0\rightrightarrows k)$.
 In this case, for any quiver representation $V$, one has
$\Ext_R^1(V,N)=0$ if and only if $V$ is projective.
 In particular, $\Ext_R^1(V,N)\ne0$ for any nonzero representation $V$
with $f_V$~invertible.

 If the map~$f_L$ is surjective, then $N=(L_1/\ker(f_L),L_2)$ is
a nonzero quotient representation of $C$ with $g_N=0$ and
$f_N$~invertible.
 In this case, it suffices to notice that a nonzero vector space with
a zero operator $x=g/f$ is not an injective object of the category
of $k[x]$\+modules.
 In particular, consider the quiver representation
$V=(k\rightrightarrows k)$ with $f_V=1$ and $g_V=0$ (so $f_V$~is
invertible and $g_V-\lambda f_V$ is invertible for all $\lambda\in\boX$).
 Then $\Ext^1_R(V,N)=\Ext^1_{k[x]}(k,N_1)\cong N_1\ne0$ (where
$x=g/f$~acts by zero both in $k$ and in~$N_1$).

 In both cases, we have found a $U_\boX$\+module $V$ such that
$\Ext^1_R(V,N)\ne0$.
 Since the category of Kronecker quiver representations has homological
dimension~$1$, the functor $\Ext^1_R(V,{-})$ is right exact and it
follows that $\Ext^1_R(V,C)\ne0$.
\end{proof}

\begin{lem} \label{X-implies-Y}
 For any\/ $\boX$ not containing\/~$0$, one has: \par
\textup{(a)} for any\/ $\boX$\+comodule $M=(f_M,g_M)$, the vector space
$M_1\cong M_2$ with the linear operator $y=f_M/g_M$ is
a $v_\boY$\+comodule; \par
\textup{(b)} for any\/ $\boX$\+contramodule $C=(f_C,g_C)$, the vector
space $C_1\cong C_2$ with the linear operator $y=f_C/g_C$ is
a $v_\boY$\+contramodule.
\end{lem}

\begin{proof}
 Part~(b): For any quiver representations $V$ and $C$ with the operators
$g_V$ and $g_C$ invertible one has $\Ext^*_R(V,C)\cong
\Ext^*_{k[y]}(V_1,C_1)$, where $y$~acts in $V_1\cong V_2$ and
$C_1\cong C_2$ by the operators~$f/g$.
 Set $V_1=T_\boY=V_2$, with the operator~$f_V$ being the multiplication
with~$y$ and $g_V=\id$.
 Then the maps $f_V$ and $g_V$ are invertible, and so is the map
$g_V-\lambda f_V$ for all $\lambda\in\boX$.
 Hence $\Ext^*_{k[y]}(T_\boY,C_1)\cong\Ext^*_R(V,C)=0$ whenever
$C$ is an $\boX$\+contramodule.
 The proof of part~(a) is similar.
\end{proof}

\begin{lem} \label{Y-implies-X}
 For any\/ $\boX$ not containing\/~$0$, one has: \par
\textup{(a)} any Kronecker quiver representation $M=(f_M,g_M)$ such
that the map~$g_M$ is invertible and the vector space $M_1\cong M_2$
with the operator $y=f_M/g_M$ belongs to $k[y]\modl_{v_\boY\co}$ is
an\/ $\boX$\+comodule; \par
\textup{(b)} any Kronecker quiver representation $C=(f_C,g_C)$ such
that the map~$g_C$ is invertible and the vector space $C_1\cong C_2$
with the operator $y=f_C/g_C$ belongs to $k[y]\modl_{v_\boY\ctra}$ is
an\/ $\boX$\+contramodule.
\end{lem}

\begin{proof}
 Part~(b): the class of $\boX$\+contramodules is closed under infinite
products in the category of Kronecker quiver representations.
 Hence, in view of Lemma~\ref{polynomials-co-contra}(b), it suffices
to consider the case when the vector space $C_1\cong C_2$ with
the operator $y=f_C/g_C$ belongs to $k[y]\modl_{v_{\{\mu\}}\ctra}$
for some fixed value of $\mu\in\boY$.
 Changing the coordinate on $\boP^1(k)$ reduces the question to
the case $\mu=0$.

 Furthermore, any $v_{\{0\}}$\+contramodule (or in other words,
a $k$\+vector space with a $y$\+power infinite summation operation)
can be obtained from the $1$\+dimensional vector space~$k$ with
the operator $y=0$ using cokernels, extensions, and projective limits
(of which the latter reduce to kernels and infinite products).
 The class of $\boX$\+contramodules is closed under all those
operations in the category of quiver representations; so it suffices
to show that the representations $C$ with $g_C$ invertible and
$f_C=0$ are $\boX$\+contramodules for all $\boX\ni\infty$.
 Without loss of generality, one can assume that $\boX=\{\infty\}$.

 Let $V$ be a Kronecker quiver representation with $f_V$ invertible.
 We have to check that $\Hom_R(V,C)=0=\Ext^1_R(V,C)$.
 Indeed, any morphism $h\:V\rarrow C$ vanishes, since the invertibility
of~$f_V$ and the vanishing of~$f_C$ together imply the vanishing of
the map $h_2\:V_2\rarrow C_2$, and then in view of the invertibility
of~$g_C$ the map $h_1\:V_1\rarrow C_1$ vanishes as well.
 Now let $0\rarrow C\rarrow A\rarrow V\rarrow0$ be a short exact
sequence of quiver representations.
 Then the subspaces $C_2$ and $f_A(A_1)$ form a direct sum decomposition
of the vector space~$A_2$, and it follows that the subspaces $C_1$ and
$g_A^{-1}(f_A(A_1))$ form a direct sum decomposition of the vector
space~$A_1$.
 Thus the short exact sequence of representations is split.
 
 The proof of part~(a) is similar.
\end{proof}

\begin{proof}[Proof of Proposition~\ref{outside-of-zero-prop}]
 Follows from Lemmas~\ref{implies-invertibility},
\ref{X-implies-Y}, and~\ref{Y-implies-X}.
\end{proof}

 Now we consider the general case when the subset $\boX\subset\boP^1(k)$
may contain the point~$0$ (so one can possibly have $\boX=\boP^1(k)$).
 The idea is to compute the $R$\+$R$\+bimodule $K=U_\boX/R$, and
consequently the functors $\Gamma_{u_\boX}=\Tor_1^R(K,{-})$ and
$\Delta_{u_\boX}=\Ext^1_R(K,{-})$.
 Then we will use the following category-theoretic observations.
 
\begin{lem}
 Let\/ $\sC$ be a category with products and a zero object.
 Assume that the identity endofunctor\/ $\Id_\sC\:\sC\rarrow\sC$
decomposes as a product of a family of functors $F_i\:\sC\rarrow\sC$,
where $i$~ranges over some index set~$I$,
$$
 \Id_\sC=\prod\nolimits_{i\in I} F_i.
$$
 Denote by\/ $\sC_i\subset\sC$ the essential image of the functor $F_i$,
viewed as a full subcategory in\/~$\sC$.
 Then the category\/ $\sC$ is equivalent to the Cartesian product of
the categories\/ $\sC_i$,
$$
 \sC\,\cong\mathop{\text{\huge $\times$}}\nolimits_{i\in I}\sC_i.
$$
\end{lem}

\begin{proof}
 The functor $F\:\sC\rarrow
\mathop{\text{\Large $\times$}}_{i\in I}\sC_i$ is simply the collection
of the functors $F_i$, that is $F=(F_i)_{i\in I}$.
 The inverse functor $G\:\mathop{\text{\Large $\times$}}_{i\in I}\sC_i
\rarrow\sC$ is the functor of $I$\+indexed product in the category $\sC$
restricted to the full subcategory
$\mathop{\text{\Large $\times$}}_{i\in I}\sC_i\subset\sC^I$.
 Clearly, the composition $G\circ F\:\sC\rarrow\sC$ is the identity
functor.

 The key observation is that for any objects $C$, $D\in\sC$, any morphism
$\prod_{i\in I}F_i(C)\cong C\rarrow D\cong\prod_{i\in I}F_i(D)$
decomposes as a product of morphisms $F_i(C)\rarrow F_i(D)$.
 It follows that, for any objects $C_i\in\sC_i$ and $D_j\in\sC_j$ with
$i\ne j$, there are no nonzero morphisms $C_i\rarrow D_j$.
 Hence for any object $C_i\in\sC_i$ one has $F_j(C_i)=0$ for all $j\ne i$,
and consequently $F_i(C_i)=C_i$.
 Furthermore, the functors $F_i$ preserve products in $\sC$, since
they are retracts of the identity functor.
 This allows to show that the composition $F\circ G$ is the identity
functor.
\end{proof}

 We recall that, in the category-theoretic terminology, a left adjoint
functor to the inclusion of a subcategory is called a \emph{reflector},
and a subcategory admitting such an adjoint functor is said to be
\emph{reflective}.

\begin{lem} \label{reflector-decomposed}
 Let\/ $\sA$ be a category with products and a zero object, and let\/
$\sC\subset\sA$ be a reflective full subcategory with the reflector\/
$\Delta\:\sA\rarrow\sC$.
 Suppose that the functor $\Delta$ decomposes as a product of a family
of functors\/ $\Delta_i\:\sA\rarrow\sC$,
$$
 \Delta=\prod\nolimits_{i\in I}\Delta_i.
$$
 Denote by\/ $\sC_i\subset\sA$ the essential image of
the functor\/ $\Delta_i$, viewed as a full subcategory in\/~$\sA$.
 Then one has\/ $\sC_i\subset\sC$, the category\/ $\sC$ is equivalent to
the Cartesian product of the categories\/ $\sC_i$, and the functor\/
$\Delta_i\:\sA\rarrow\sC_i$ is the reflector onto the full subcategory\/
$\sC_i\subset\sA$.
\end{lem}

\begin{proof}
 Set $F_i=\Delta_i|_\sC$ and apply the previous lemma.
\end{proof}

 The following theorem is the main result of this section.

\begin{thm}
 For any subset\/ $\infty\in\boX\subset\boP^1(k)$,
the following assertions hold. \par
\textup{(a)} Any\/ $\boX$\+comodule $M$ has a unique, functorial
decomposition into a direct sum of\/ $\{\lambda\}$\+comodules $M_\lambda$
over\/ $\lambda\in\boX$, and any such direct sum
$M=\bigoplus_{\lambda\in\boX}M_\lambda$ of\/ $\{\lambda\}$\+comodules
$M_\lambda$ is an\/ $\boX$\+comodule.
 The category of\/ $\boX$\+comodules is thus equivalent to the Cartesian
product of the categories of\/ $\{\lambda\}$\+comodules over
$\lambda\in\boX$ (each of which is equivalent to the category of
$k$\+vector spaces with a locally nilpotent linear operator~$z$). \par
\textup{(b)} Any\/ $\boX$\+contramodule $C$ has a unique, functorial
decomposition into a direct product of\/ $\{\lambda\}$\+contramodules
$C^\lambda$ over $\lambda\in\boX$, and any such direct product
$C=\prod_{\lambda\in\boX}C^\lambda$ of\/
$\{\lambda\}$\+contramodules $C^\lambda$ is an\/ $\boX$\+contramodule.
 The category of\/ $\boX$\+contramodules is thus equivalent to the
Cartesian product of the categories of\/ $\{\lambda\}$\+contramodules
over $\lambda\in\boX$ (each of which is equivalent to the category of
$k$\+vector spaces with a $z$\+power infinite summation operation).
\end{thm}

\begin{proof}
 The $R$\+$R$\+bimodule $U_\boX$ can be described as the following
representation of the Cartesian square $(\bullet\rightrightarrows\bullet)
\times(\bullet\rightrightarrows\bullet)$ of the Kronecker quiver
$\bullet\rightrightarrows\bullet$
(we recall the notation $S_\boX=\boX^{-1}k[x]$ for the relevant
localization of the polynomial ring $k[x]$):
\begin{equation}
\begin{tikzcd}
 S_\boX \arrow[rr, bend left=10, "1"] \arrow[rr, bend right=10, "x"'] &&
 S_\boX \\ \\
 S_\boX \arrow[rr, bend left=10, "1"] \arrow[rr, bend right=10, "x"']
 \arrow[uu, bend left=10, "1"] \arrow[uu, bend right=10, "x"'] &&
 S_\boX \arrow[uu, bend left=10, "1"] \arrow[uu, bend right=10, "x"']
\end{tikzcd}
\end{equation}

In the same vein, the $R$\+$R$\+bimodule $K=U_\boX/R$ is described as
the following representation of the quiver $(\bullet\rightrightarrows
\bullet)\times(\bullet\rightrightarrows\bullet)$:
\begin{equation} \label{U/R-representation}
\begin{tikzcd}
 S_\boX/k
 \arrow[rr, bend left=10, "1"] \arrow[rr, bend right=10, "x"'] &&
 {S_\boX/(k\oplus kx)\mkern-24mu} \\ \\
 S_\boX \arrow[rr, bend left=10, "1"] \arrow[rr, bend right=10, "x"']
 \arrow[uu, bend left=15, "1"] \arrow[uu, bend right=15, "x"'] &&
 S_\boX/k \arrow[uu, bend left=15, "1"] \arrow[uu, bend right=15, "x"']
\end{tikzcd}
\mkern24mu
\end{equation}
 
 The key observation is that
the representation~\eqref{U/R-representation} decomposes into a direct
sum of representations indexed by the points of the set~$\boX$.
 This direct sum decomposition of~\eqref{U/R-representation} is induced
by the direct sum decomposition
$$
 S_\boX\,\cong\, k[x]\,\oplus\,\bigoplus\nolimits
 _{\lambda\in\boX\setminus\{\infty\}}
 \left(\bigoplus\nolimits_{n\ge1}k(x-\lambda)^{-n}\right)
$$
of the vector space~$S_\boX$.
 Hence we obtain an $\boX$\+indexed direct sum decomposition of
the functor $\Gamma_{u_\boX}=\Tor_1^R(K,{-})$ and an $\boX$\+indexed
direct product decomposition of the functor
$\Delta_{u_\boX}=\Ext^1_R(K,{-})$.
 It remains to apply Proposition~\ref{u-comodule-category}(b) together
with the dual assertion to Lemma~\ref{reflector-decomposed} in order
to deduce part~(a) of the theorem, and
Proposition~\ref{u-contramodule-category}(b) together with
Lemma~\ref{reflector-decomposed} in order to deduce part~(b).

 The description of the categories of $\{\lambda\}$\+comodules and
$\{\lambda\}$\+contramodules in parts~(a) and~(b) is provided by
Proposition~\ref{outside-of-zero-prop}.
 It suffices to change the coordinate on the projective line
$\boP^1(k)$ suitably in order to include the case $\lambda=0$.
\end{proof}

\end{document}